\documentclass[12pt]{article}
\usepackage{amsmath,amsfonts,amssymb,amsthm,amscd}
\title{Generically stable and smooth measures in NIP theories}
\date{Feb. 23, 2010}
\author{Ehud Hrushovski\thanks{Supported by ISF grant 1048/07}\\Hebrew University of Jerusalem\and Anand Pillay\thanks{Supported
by a Marie Curie Chair EXC 024052 and EPSRC grant EP/F009712/1}\\University of Leeds \and Pierre Simon\\ENS and Univ. Paris-Sud 11}
\newtheorem{Theorem}{Theorem}[section]
\newtheorem{Proposition}[Theorem]{Proposition}
\newtheorem{Definition}[Theorem]{Definition} 
\newtheorem{Remark}[Theorem]{Remark}
\newtheorem{Lemma}[Theorem]{Lemma}
\newtheorem{Corollary}[Theorem]{Corollary}
\newtheorem{Fact}[Theorem]{Fact}

\newcommand{\R}{\mathbb R}   
\newcommand{\Q}{\mathbb Q}  
  \newcommand{\Zz}{\mathbb Z}
  \newcommand{\Nn}{\mathbb N}

\def\union{\cup}

\begin{document}
\maketitle

\begin{abstract} 
We formulate the measure analogue of generically stable types in first order theories with $NIP$ (without the independence property), giving several characterizations, answering some questions from \cite{NIPII}, and giving another treatment of uniqueness 
results from \cite{NIPII}. 
We introduce a notion of ``generic compact domination", relating it to stationarity of Keisler measures, and also giving group versions. We also prove the ``approximate definability" of arbitrary Borel probability measures on definable sets in the real and $p$-adic fields.
\end{abstract}

\section{Introduction and preliminaries}

In this introduction we will discuss background and motivation, describe and summarize our main results, and then recall some essential definitions and
prior results. A familiarity with the earlier papers \cite{NIPI} and \cite{NIPII}  would  be advantageous, but we will try to make the bulk  of the paper accessible to a wider audience, even though we are somewhat terse. Even in the introduction we may make some rather advanced comments or references, and the general reader should feel free to ignore these at least on the first reading. 

Shelah defined a formula $\phi(x,y)$ to have the {\em independence property} if there exist arbitrarily large (finite) sets
$A$ such that {\em any} subset $B$ of $A$ has the form $\{a \in A: \phi(a,b) \}$, for some parameter $b$.  A theory has
$NIP$ if no formula has the independence property.  An equivalent definition in a combinatorial / probabilistic rather than logical setting was found by Cervonenkis and Vapnik \cite{Vapnik-Chervonenkis}.  
$o$-minimal and $p$-minimal theories are notable examples.   

A general theme in this paper is
``stable-like" behaviour in theories with $NIP$.
One of the main points  is to develop the theory of ``generically stable measures" in $NIP$ theories, in analogy with generically stable types.
A ``generically stable type" is a global type (namely a complete type over a saturated model) which looks very much like a type in a stable theory, for example it is both definable over and finitely satisfiable in some small model $M$. The theory, at least in the $NIP$ context was developed in \cite{Shelah783}, \cite{NIPII} and \cite{Usvyatsov}. Among the consequences (or even equivalences) of generic stability of a type $p$, assuming $T$ has $NIP$, are {\em nonforking symmetry} (or the total indiscernibility of any ``Morley sequence" in $p$), as well as {\em stationarity}, in the sense that 
$p$ is the {\em unique} global nonforking extension of its restriction to $M$.

In the theory of algebraically closed valued fields generically stable types coincide with stably dominated types and play a major role in the structural analysis of definable sets \cite{HHM2} as well as in a model-theoretic approach to Berkovich spaces \cite{Hrushovski-Loeser}.
However in  $o$-minimal theories and $p$-adically closed fields for example, there are 
{\em no} (nonalgebraic) generically stable types. 

On the other hand, what we have called Keisler measures (introduced in Keisler's seminal paper \cite{Keisler1}), are the natural generalization of complete types to finitely additive $[0,1]$ valued measures on Boolean algebras of definable sets.   Keisler showed (in slightly different terms) that in a NIP theory,  for any Keisler measure $\mu$ on a model $M$
any formula $\phi(x,y)$ and any $\epsilon>0$, there exist finitely many formulas $\phi(x,b_i)$ such that 
for any $b$, $\mu( \phi(x,b_i) \triangle \phi(x,b)) < \epsilon$ for some $i$.  To see this, take a maximal set $\{b_i\}$
such that $\mu(\phi(x,b_i) \triangle \phi(x,b_j)) \geq \epsilon/2$ for $i \neq j$.  If this set is finite, we are done.
If it is infinite, by compactness one obtains an indiscernible sequence $(b_n: n \in \Nn)$ and some measure $\mu'$
with the same property. 
So $\mu(\phi(x,b_m) \setminus \phi(x,b_{m+1})) \geq \epsilon/4 $ for all odd $m$ (or for all even $m$; say odd.) 
 It follows by elementary measure considerations that 
$(\mu(\phi(x,b_m) \setminus \phi(x,b_{m+1})): m=2,4,\ldots)$ cannot be $k$-inconsistent, for any $k$.  So
$\{\phi(x,b_m): m=1,2,\ldots \} \union \{\neg \phi(x,b_m): m=2,4,\ldots \}$ is consistent.  But by indiscernibility
the same must be true for any subset in place of the odds, contradicting $NIP$.

Keisler measures play an important role in the solution of certain conjectures on groups in 
$o$-minimal structures \cite{NIPI}. They were studied further and from a more stability-theoretic point of view in \cite{NIPII}.  In fact in the latter paper, we
{\em defined} generically stable measures to be global Keisler measures which are both definable over and finitely satisfiable in some small model. We also found natural examples as translation invariant measures on suitable definable groups (such as definably compact groups in $o$-minimal theories). However, there were on the face of it technical obstacles to obtaining analogous properties (like stationarity, total indiscernibility) for generically stable measures as for generically stable types. For example, what is a ``realization" of a measure, or a ``Morley sequence in a measure"? This is solved in various ways in the current paper, including making heavy use of Keisler's ``smooth measures" (see section 2). Essentially a complete counterpart to the type case is obtained, the main results along these lines being Theorem 3.2 and Proposition 3.3, where another property, ``frequency interpretation measure" makes an appearance. Moreover we also point out how widespread generically stable measures are in $NIP$ theories. 

Let us take the opportunity to remark that a natural formal way to deal with ``technical" issues such as realizing Keisler measures would be to  pass to the randomization $T^{R}$ of $T$. $T^{R}$ is a continuous first order theory whose models are random variables in models of $T$. The type spaces of $T^{R}$ correspond to the spaces of Keisler measures (over $\emptyset$) of $T$.
This randomization was introduced by Keisler and situated in the context of continuous logic by Ben Yaacov and Keisler \cite{BY-Keisler}. Ben Yaacov proved that $T^{R}$ has $NIP$ if $T$ does, and further showed  that making systematic use of $T^{R}$ would provide, in principle, another route to the results of the current paper (\cite{BY1}, \cite{BY2}).   Measures in $NIP$ theories are roughly of the same complexity as types, as is evidenced for instance by boundedness of the number of formulas modulo measure zero.
But measures on the space of measures appear  to be genuinely analytic objects, and required nontrivial analytic tools in Ben Yaacov's treatment.  We make use of a weak version of  Ben Yaacov's preservation theorem (see Lemma 2.10) to give one proof of our characterization of generically stable measures
(Theorem 3.2), but also give an independent proof remaining within the usual model theoretic framework.

In section  4 we generalize the notion of {\em a group with finitely satisfiable generics} or with the $fsg$ property, to types and measures, and make the connection with generic stability.

In section 5 we introduce a weak notion of ``compact domination" where the set being dominated is a space of types rather than a definable or type-definable set. We relate this to stationarity of measures (unique nonforking extensions) in what we consider to be a measure-theoretic version of the finite equivalence theorem.

In section 6, we prove smoothness (unique extension to a global Keisler measure)
of Borel probability measures on real or $p$-adic semialgebraic sets, yielding 
a quite extensive strengthening of work by Karpinski and Macintyre in the case of Haar measure.

As far as sections 2, 3 and 6 are concerned, the paper is relatively self-contained. However sections 4 and 5 make rather more references to the earlier papers \cite{NIPI} and \cite{NIPII}, and not only hyperimaginaries but also the compact Lascar group are involved.

The current paper does not only follow on from those two earlier papers, but also naturally continues and builds on Keisler's original papers \cite{Keisler1}, \cite{Keisler2}.

\vspace{2mm}
\noindent
We fix a complete first order theory $T$. We typically work in $T^{eq}$. For convenience we choose a very saturated ``monster model" or ``universal domain" ${\bar M}= {\bar M}^{eq}$. $M, N, M_{0},..$ denote small elementary submodels. For now $A,B,C,..$ denote subsets, usually small, 
of ${\bar M}$. $x,y,..$ range over (finitary) variables and by convention a variable carries along with it its sort. 

The reader is referred to say \cite{Pillay-book}, \cite{Adler}, \cite{Poizat}, \cite{Shelah} as well as  
\cite{NIPI}, \cite{NIPII}, for extensive and detailed material around stable theories, $NIP$ theories as well as the adaptation/interpretation of forking to types and measures in $NIP$ theories.

However we recall here the key notions relevant to the current paper. 

It is convenient to start with the notion of a finitely additive measure $\mu$ on an arbitrary  Boolean algebra $\Omega$: $\mu(b)\in [0,1]$ for all $b$ in $\Omega$, $\mu(1) = 1, \mu(0) = 0$ and $\mu$ is finitely additive. As in section 4 of \cite{NIPII}, such a measure  on a Boolean algebra $\Omega$ can be identified with a regular Borel probability measure on the Stone space $S_{\Omega}$ of $\Omega$. The set of finitely additive measures  on $\Omega$ is naturally a compact space.

We apply this to our monster model ${\bar M}$. 
By a Keisler measure $\mu_{x}$ {\em over $A$} we mean a finitely additive measure on the Boolean algebra of formulas $\phi(x)$ over $A$ up to equivalence in ${\bar M}$.
So a Keisler measure over $A$ generalizes the notion of a complete type over $A$ rather than a partial type over $A$. By a global Keisler measure we mean one over ${\bar M}$.
So again a global Keisler measure generalizes the notion of a global complete type. We repeat from the previous paragraph that a Keisler measure $\mu_{x}$ over $A$ coincides with a regular probability measure on $S_{x}(A)$. We often talk about closed, open, Borel, sets, over $A$. So for example, a
Borel set over $A$ is simply the union of the sets of realizations in ${\bar M}$ of types $p\in S(A)$, for $p$ in some given Borel subset of $S_{x}(A)$. 

Keisler \cite{Keisler1} uses the notion of a measure over or on a {\em fragment}, which it is now convenient to work with in a generalized form. By a fragment $F$ he means a small collection of formulas $\phi(x)$ (or definable sets of sort $x$) which is closed under (finite) Boolean combination.   (A typical case is the collection of all formulas over a given base set $A$.)  Then a Keisler measure on or over $F$ is simply a finitely additive probability measure on this Boolean algebra of definable subsets of sort $x$. As above this identifies with a regular Borel probability measure on the space $S_{F}$ of complete types over $F$. He also remarks that if $F\subseteq G$ are fragments (in sort $x$) then any Keisler measure on $F$ extends to one on $G$. In particular any Keisler measure on $F$ extends to a global Keisler measure on the sort of $x$. 

For most of this paper this notion of fragment is adequate, and the reader may proceed with this in mind,
at least until section 5.  However in some situations we will need to consider  algebras of subsets of  ${\bar M}$  that, while contained in the Borel subalgebra of $S_F$ for various fragments $F$ of formulas, cannot canonically be presented in this manner.  We 
therefore give in advance a formalism beginning with closed rather than clopen sets,
i.e. partial types rather than formulas.   Our fragments correspond to  small topological quotients of the space of global types:   an element  of the fragment is the pullback of a closed set.  We describe this more syntactically in 
the next paragraph.  
 
Let $F$ now consist of a small collection of partial types $\Sigma(x)$ in a fixed set of variables $x$, identified if you wish with their sets of realizations in ${\bar M}$.  We assume $F$ is closed under finite disjunctions and (possibly infinite) conjunctions. We will call a subset of the $x$-sort of ${\bar M}$  {\em closed} over $F$ it is defined by a partial type in $F$, and open over $F$ if it is the complement of a closed over $F$ set (and also we can obtain the Borel over $F$ sets). 

\begin{Definition} 
(a) Let $F$ be as in the above paragraph. We call $F$ a {\em fragment} if
\newline
(i) any open set over $F$ is a union of closed sets over $F$, and
\newline
(ii) any two disjoint closed over $F$ sets are separated by two disjoint open over $F$ sets.
\newline
(b) If $F$ is a fragment, let $S_{F}$ denote the set of maximal partial types in $F$ (i.e. maximal among partial types in $F$). 
\end{Definition}

Clearly a fragment in the sense of Keisler extends uniquely to a fragment in the sense of Definition 1.1.

For a fragment $F$ define a  topology  on $S_{F}$ in the obvious way: a closed set
is by definition a set of points extending a given partial type in $F$. Then with this definition it is clear that $S_{F}$ will be a compact Hausdorff space. 

\begin{Definition}
By a Keisler measure on or over a fragment $F$ we mean a map from the set of closed/open over $F$ sets to $[0,1]$ which is induced by a regular Borel probability measure on the space $S_{F}$.
\end{Definition} 

For hyperimaginaries, as well as the notion $bdd(A)$ (set of hyperimaginaries in the bounded closure of $A$) see \cite{HKP} or \cite{Wagner}.

\begin{Lemma} (i) Let $A$ be a small set of hyperimaginaries. Then the collection of partial types over $A$ is a fragment.
\newline
(ii) Let $F\subseteq G$ be fragments (in sort $x$). Then any Keisler measure over $F$ extends to a Keisler measure over $G$. 
\end{Lemma}

\vspace{2mm}
\noindent
One more definition at the level of fragments is:
\begin{Definition}  Let $\mu$ be a measure over a fragment $F$. Let $D$ be a Borel set over $F$ with positive $\mu$ measure. Then the localization $\mu_{D}$ of $\mu$ at $D$ is defined by: For any Borel $E$ over $F$, $\mu_{D}(E) = \mu(E\cap D)/\mu(D)$. 
\end{Definition}

Now we pass to forking for measures in $NIP$ theories. First, $T$ is said to have the independence property, if there is an indiscernible (over $\emptyset$) sequence $(a_{i}:i<\omega)$ and formula $\phi(x,b)$ such that $\models \phi(a_{i},b)$ for $i$ even, and $\models \neg\phi(a_{i},b)$ for $i$ odd.
We usually say that $T$ is (or has) $NIP$ if $T$ does not have the independence property. 

We recall that a formula $\phi(x,b)$ (where we exhibit the parameters) {\em divides over} a small set $A$ if there is an $A$-indiscernible sequence $(b_{i}:i<\omega)$ with $b_{0} = b$ such that $\{\phi(x,b_{i}):i<\omega\}$ is inconsistent. A formula {\em forks over} $A$ if it implies a finite disjunction of formulas each of which divides over $A$.
We say that a global Keisler measure $\mu_{x}$ does not divide (does not fork) over a small set $A$ if every formula $\phi(x)$ with positive $\mu$-measure does not divide (does not fork) over $A$. In fact for such global  $\mu$,  not dividing over $A$ and not forking over $A$ are equivalent, and $A$ can even be a set of hyperimaginaries. Recall from \cite{NIPII} that assuming $T$ has $NIP$, $\mu$ does not fork over $A$ iff $\mu$ is $Aut({\bar M}/bdd(A))$ invariant (we just say $bdd(A)$-invariant) iff $\mu$ is Borel definable over $bdd(A)$. 
Here Borel-definability of $\mu$ over $A$, means that for a given formula $\phi(x,y)\in L$ and closed  subset $C$ of $[0,1]$, $\{b: \mu(\phi(x,b)\in C\}$ is Borel over $A$.
We persist in calling a global measure $\mu_{x}$ {\em definable} over $A$ if for $\phi(x,y)\in L$ and closed $C\subseteq [0,1]$, $\{b:\phi(x,b)\in C\}$ is closed over $A$, namely type-definable over $A$. (Although the expression $\infty$-definable might be better.) We also say that $\mu$ is finitely satisfiable in $A$ (where usually $A$ is a model $M$) if every formula over ${\bar M}$ with positive $\mu$-measure is satisfied by some
element (or tuple) from $A$. These are all natural generalizations of the corresponding classical notions for global types. 

Let us make the important remark that if the global Keisler measure $\mu_{x}$ is finitely satisfiable over $A$, then it is also $A$-invariant, hence 
(assuming that $T$ has $NIP$) is Borel definable (over $A$). 

The (nonforking) {\em product} of measures $\mu_{x}$ and $\lambda_{y}$ is a fundamental notion in this paper (as well as in \cite{NIPII}). Identifying a global Keisler measure $\mu_{x}$ with a measure on $S_{x}({\bar M})$, then this could not simply be the usual product measure because the type space $S_{xy}({\bar M})$ is not the product of $S_{x}({\bar M})$ with $S_{y}({\bar M})$ (and the same issue arises for types). In the case of types, if $p(x)$, $q(y)$ are global complete types, and $p(x)$ does not split over $A$ for some small $A$ (equivalently is $A$-invariant), then we can form $p(x)\otimes q(y)$ in variables $xy$, defined as $tp(a,b/{\bar M})$ where $b$ realizes $q$ and $a$ realizes $p|{\bar M},b$. Equivalently, $\phi(x,y,m)\in p(x)\otimes q(y)$ if for some (any) $b$ realizing $q|A,m$, $\phi(x,b)\in p$. 
Now if $\mu(x)$ is Borel definable (over $A$) say, and $\lambda(y)$ arbitrary (both global say) then the analogous product $\mu_{x}\otimes \lambda_{y}$ is obtained via 
{\em integration}: Pick a formula $\phi(x,y)$ over ${\bar M}$. For any $q(y)\in S_{y}({\bar M})$, and realization $b$ of $q$, we can consider the extension $\mu' = \mu|({\bar M},b)$ of $\mu'$ given by applying the same Borel definition. In any case $\mu'(\phi(x,b))$ depends only on $q$, so we can write it as $f(q)$ for some function $f:S_{y}({\bar M})\to [0,1]$.  The Borel definability of $\mu$ says that the function $f$ is Borel (preimage of a closed set is Borel). Hence we can integrate $f$ along $\lambda$ (treated as a Borel measure on $S_{y}({\bar M})$), to obtain $\int_{S_{y}({\bar M})} f(q) d\lambda$. And we call this $(\mu(x)\otimes \lambda(y))(\phi(x,y))$. 

Note that this integral can be ``computed" as follows: again choose a formula $\phi(x,y,m)$ where now we exhibit additional parameters from ${\bar M}$ as $m$. Fix natural number $N$ and partition $[0,1]$ into equal intervals $I_{1},..,I_{N}$ of length $1/N$, let $Y_{j} = \{b:\mu(\phi(x,b,m)\in I_{j}\}$ (a Borel set over $A,m$), let $c_{j}$ be the midpoint of $I_{j}$. Let $F_{N} = \sum_{j=1,..,N}\lambda(Y_{j})c_{j}$. Then 
$(\mu(x)\otimes \lambda(y))(\phi(x,y,m)) = lim_{N\to\infty}F_{N}$.

We will often use this, when doing approximations or computations. 

\begin{Lemma}  Suppose that $\mu(x)$, $\lambda(y)$ are global Keisler measures which are both
definable. Then so is $\mu(x)\otimes \lambda(y)$. Likewise for Borel definable, and (assuming $NIP$) ``finitely satisfiable in a small model". 
\end{Lemma}
\begin{proof} Let us just deal with the finitely satisfiable case, the proof of which will be an elementary example of methods which pervade the paper. Assume that both $\mu$ and $\lambda$ are finitely satisfiable in $M$. We show that $\mu\otimes\lambda$ is too. Let $\phi(x,y,m)$ be a formula over ${\bar M}$ with positive $\mu_{x}\otimes \lambda_{y}$ measure (where we exhibit the parameter $m$). It follows from the definition of this ``nonforking product" that $Y = \{b\in {\bar M}:\mu(\phi(x,b,m)) > 0\}$ is a Borel set over $M,m$ of positive $\lambda_{y}$-measure. By regularity of $\lambda$ (as a Borel measure on $S_{y}(M,m)$) there is a closed over $M,m$ set $Z$ say, with $Z\subseteq Y$ and $\lambda(Z) > 0$. By compactness let $b\in Z$ be weakly random for $\lambda|(M,m)$ in the sense that $\models\neg\chi(b)$ for any formula $\chi(y)$ over $M,m$ with $\lambda(\chi(y)) = 0$. As $b\in Z\subseteq Y$, $\mu(\phi(x,b,m)) > 0$. As $\mu$ is finitely satisfiable in $M$, there is $a'\in M$ such that $\models\phi(a',b,m)$. By choice of $b$, $\lambda(\phi(a',y,m)) > 0$, so by finite satisfiability of $\lambda$ in $M$ there is $b'\in M$ such that $\models \phi(a',b',m)$. This completes the proof.

\end{proof}

In general, the product of measures is not commutative; a measure need not even commute with itself:
we can have $\mu_x \otimes \mu_y \neq \mu_y \otimes \mu_x$.  The question of commutativity will become
central later on.  We note at this point that  $\mu_x \otimes \lambda_y = \lambda_y \otimes \mu_x$ iff the Borel measure-zero sets of these two measures coincide.  This will not be explicitly used in the body of the paper. In the 
lemma below we take the point of view of a global Keisler measure as a regular probability measure on the relevant Stone space of global types.

\begin{Lemma} (NIP)  Let $\mu_x,\lambda_y$ be global measures, invariant over some small set.

\begin{enumerate} \item For any definable set $\phi(x,y)$ there is a Borel subset $U_{\phi}$ of the space 
$S_{x}({\bar M})\times S_{y}({\bar M})$ (so in the $\sigma$-algebra on $S_{xy}({\bar M})$ generated by rectangles $D_{x}\times E_{y}$)
 such that $\phi(x,y),U_\phi$ are equal up to $\mu_x \otimes \lambda_y$-measure zero.

\item Commutativity can be checked at the level of the Borel measure-zero ideal:
if $\mu_x \otimes \lambda_y (U)=0$ for any closed $U$ such that $\lambda_y \otimes \mu_x (U)=0$,
then $\mu_x \otimes \lambda_y = \lambda_y \otimes \mu_x$.

\end{enumerate}

\end{Lemma}

\begin{proof}     We may write $\mu = \int_{a \in X} p^a, \lambda=\int_{b \in Y} q^b$, where
$X,Y$ are the Stone spaces of the Boolean algebra of global definable sets, modulo the measure zero sets
of $\mu,\lambda$ respectively; made into measure spaces using the measures induced from $\mu,\lambda$;
where for $a \in X, b\in Y$, $p^a,q^b$ are the corresponding invariant types.  
\begin{enumerate} 
\item 
Given a formula $\phi(x,y)$,
let $U_\phi = \{(a,b) \in X \times Y:  \phi \in p^a \otimes q^b \}$.  We will show below that $U_\phi$ is  Borel up to a measure zero set.   Clearly the stated equality holds.
\item The assumption extends from closed to Borel sets:  if $U'$ is any Borel set with $\mu_x \otimes \lambda_y (U')=0$, then $\mu_x \otimes \lambda_y (U)=0$ for all closed $U \subseteq U'$, so by assumption
 $\lambda_y \otimes \mu_x (U)=0$ for all such $U$, and since this measure is regular, 
  $\lambda_y \otimes \mu_x (U')=0$.  
 Now let $\phi(x,y)$ be any formula.   By (1) there exists a Borel $U$ with 
    $\mu_x \otimes \lambda_y (\phi \triangle U) =0$, so $ \lambda_y \otimes \mu_x(\phi \triangle U)=0$.
But $\mu_x \otimes \lambda_y (U)  =  \lambda_y \otimes \mu_x(U)$.  So
 $\mu_x \otimes \lambda_y (\phi)=  \lambda_y \otimes \mu_x(\phi)$.
\end{enumerate}

To show that $U_\phi$ is Borel up to measure $0$, choose a finite set $\mathcal{L}_m$ of formulas $\phi(x,e)$ such that for any parameter $c$, there exists a definable set $D \in \mathcal{L}_{m}$ with $\mu(\phi(x,c) \triangle D)<  2^{-m}$.  
Let  $\mathcal{L} = \union_m \mathcal{L}_m$,  and fix some enumeration of $\mathcal{L}$ (or just of each  $\mathcal{L}_m$).  
All formulas of $\mathcal{L}$ are defined over some small model $M_0$, such that $\mu,\lambda$ are
$M_0$-invariant.  

Any $b \in Y$ determines a weakly random type for $\lambda$ over $M_0$,   $q^b
|M_0$.  Since $\mu$ is $M_0$-invariant,
for $c,c' \models q^b|M_0$ and $D \in L(M)$  we have $\mu(\phi(x,c) \triangle D)<  2^{-m}$  iff
$\mu(\phi(x,c') \triangle D)<  2^{-m}$; so $\mu( \phi(x,c) \triangle \phi(x,c')) = 0$.  Thus we will write
$\phi(x,b)$ to denote the class of any such $\phi(x,c)$, up to $\mu$ measure $0$.  

Let $D_m(b)$ be the least $D \in \mathcal{L}_m$  such that 
$\mu(\phi(x,b) \triangle D)<  2^{-m}$.

By the usual proof of completeness of $L^1(\mu)$,  $\phi(x,b)$ differs by $\mu$-measure zero from the Borel set
$$D(b) = \{x: (\exists m_0)(\forall m \geq m_0)(x \in D_m(a)) \}$$

Since $\mu$ is a Borel measure, 
$$ \{(b,m,D): D \in \mathcal{L}_m,  \mu(\phi(x,b) \triangle D)<  2^{-m}   \} $$
is Borel, and so the map $(b,m) \mapsto D_m(b)$ is Borel.

So $E=\{(a,b): (\exists m_0)(\forall m \geq m_0)(\exists D \in \mathcal{L}) (D=D_m(b)  \hbox{ and } a \in D \}$
 is also Borel.  For all $b$, $E(a)$ differs from $U_\phi(b)$ by $\mu$-measure zero.
 This finishes the proof.  
\end{proof}

\begin{Definition} (Assume $NIP$.) Let $\mu_{x}$ be a global Keisler measure which is invariant (i.e. $A$-invariant for some small $A$). Then $\mu^{(n)}_{x_{1},..,x_{n}}$ is defined (inductively) by $\mu^{(1)}_{x_{1}} = \mu_{x_{1}}$, and $\mu^{(n+1)}_{x_{1},..,x_{n},x_{n+1}} = \mu_{x_{n+1}} \otimes \mu^{(n)}_{x_{1},..,x_{n}}$.
We put $\mu^{\omega}_{x_{1},x_{2},...}$ to be the union.
\end{Definition}

\vspace{2mm}
\noindent
Finally we recall the {\em weak law of large numbers} in the form we will use it. Any basic text on probability theory is a reference.

\begin{Fact}  Let $\mu$ be a Borel probability measure on a space $X$. Let $\mu^{k}$ be the product measure on $X^{k}$. Let $Y$ be a measurable subset of $X$. For $p_{1},..,p_{k}\in X$ let $Fr_{k}({\bar p},Y)$ be $|\{i: p_{i}\in Y\}|/k$.
THEN for any $\epsilon > 0$, 
$\mu^{k}(\{(p_{1},..,p_{k}): |Fr_{k}({\bar p},Y) - \mu(Y)| < \epsilon\}) \to 1$ as $k \to \infty$.
\end{Fact}

\vspace{5mm}
\noindent
Thanks to the Wroclaw model theory group, in particular H. Petrykowski, for pointing out some errors in an early version of \cite{NIPII}, which we deal with in section 5 of the current paper. Some of the results in sections 2 and 3 of the present paper appear in the third authors Master's Thesis \cite{Simon}. However we do not follow the ``formal points" formalism from there.

\section{Smooth measures and indiscernibles}
Here we discuss smooth measures, using and repeating some material from Keisler's paper \cite{Keisler1}, but also applying the results in the context of $NIP$ theories to obtain useful results about arbitrary measures as well as ``indiscernible measures". 

We will NOT make a blanket assumption that $T$ has $NIP$. 
\begin{Definition}  A global Keisler measure $\mu_{x}$ is said to be smooth if $\mu$ is the unique global extension of $\mu|M$
for some small model $M$. We may also call $\mu$ smooth over $M$, and also call $\mu|M$ smooth. 
\end{Definition}

We should mention that Keisler's notion of a smooth measure was somewhat weaker. He called a Keisler measure over a small set (or even a fragment) if it had a unique global extension modulo the ``stable part". Possibly ``minimal" might be a better expression for us, but we stick with our Definition above. 
A key result of Keisler is Theorem 3.16 from \cite{Keisler1}:

\begin{Lemma} (Assume $T$ has $NIP$.) If $\mu_{x}$ is a Keisler measure over $M$ then it has an extension to a smooth global Keisler measure.
\end{Lemma}

Note that a complete type (over $M$, or ${\bar M}$) is smooth iff it is realized (in $M$, ${\bar M}$) respectively. A key point of the current paper (also implicit in \cite{NIPII}), is that in an $NIP$ context, one can usefully view a smooth extension of $\mu$ as a ``realization" of $\mu$, and thus deal effectively with technical issues around measures.

\begin{Lemma} Suppose $\mu_{x}$ is smooth over $M$. Let $\phi(x,y)\in L$ and $\epsilon> 0$. Then there are formulas
$\nu_{i}^{1}(x), \nu_{i}^{2}(x)$ for $i=1,..,n$ and $\psi_{i}(y)$ for $i=1,..,n$, all over $M$ such that
\newline
(i) the formulas $\psi_{i}(y)$ partition $y$-space,
\newline
(ii) for all $i$, if $\models \psi_{i}(b)$, then $\models \nu_{i}^{1}(x) \rightarrow \phi(x,b)\rightarrow \nu_{i}^{2}(x)$, and
\newline
(iii) for each $i$, $\mu(\nu_{i}^{2}(x)) - \mu(\nu_{i}^{1}(x)) < \epsilon$.
\end{Lemma}
\begin{proof} By smoothness of $\mu$ and Lemma 1.3 (iv) of \cite{Keisler1} for example, for each $b\in {\bar M}$ there are formulas $\nu^{1}(x)$, $\nu^{2}(x)$ over $M$, such that
\newline
(*)
$\models \nu^{1}(x) \rightarrow \phi(x,b) \rightarrow \nu^{2}(x) $, and $\mu(\nu^{2}(x))- \mu(\nu^{1}(x)) < \epsilon$. 
\newline
By compactness, there are
finitely many such pairs, say, $(\nu_{i}^{1}(x), \nu_{i}^{2}(x))$ such that for every $b$ one of these pairs satisfies (*). 
It is then easy to find the $\psi_{i}(y)$. 
\end{proof}

Note that it follows from Lemma 2.3 that if $\mu$ is a global smooth Keisler measure, then $\mu$ is smooth over some model $M_{0}$ of cardinality at most $|T|$. Note also that Lemma 2.3 yields a direct way of seeing both the definability over $M$ and finite satisfiability in $M$ of $\mu$.

\begin{Definition}
If $\mu_{x}$ and $\lambda_{y}$ are both Keisler measures over $M$ (with $x,y$ disjoint tuples of variables), then a Keisler measure
$\omega_{x,y}$ over $M$ extending both $\mu_{x}$ and $\lambda_{y}$ is said to be a {\em separated} amalgam of $\mu_{x}$ and $\lambda_{y}$, if for any formulas $\phi(x), \psi(y)$ over $M$, $\omega(\phi(x)\wedge \psi(y)) = \mu(\phi(x))\cdot\lambda(\psi(y))$. 
\end{Definition}

This is the same thing as saying that $\omega_{x,y}$, as a regular Borel probability measure on $S_{xy}(M)$ extends the product measure $\mu_{x} \times \lambda_{y}$ on the space $S_{x}(M)\times S_{y}(M)$. Note that as soon as $\mu_{x}$ is not a complete type, there will  be at least two extensions of $\mu_{x}\cup\mu_{y}$ to Keisler measures over $M$; one giving $x=y$ measure $1$, which will not be separated, and one extending the product $\mu_{x}\times \mu_{y}$ which
will be separated.  On the other hand if $\mu_{x}$ is a measure over $M$ and $q(y)$ a complete type over $M$ then any amalgam $\omega_{xy}$ of $\mu$ and $q$ will be separated.
\newline

We now give several corollaries of Lemma 2.3.

\begin{Corollary} Suppose $\mu_{x}$ is a smooth global Keisler measure. Then for any global Keisler measure $\lambda_{y}$, there is a unique separated amalgam of $\mu_{x}$ and $\lambda_{y}$.
\end{Corollary}
\begin{proof} Assume $\mu$ is smooth over $M$. Let $\omega_{x,y}$ be such an amalgam. Let $\phi(x,y)\in L$ and $\epsilon > 0$. Let $\nu_{i}^{1}(x)$, $\nu_{i}^{2}(x)$, $\psi_{i}(y)$, for $i=1,..,n$ be as given by Lemma 2.3. Then for each $i$, 
$\models \nu_{i}^{1}(x)\wedge\psi_{i}(y) \rightarrow \phi(x,y)\wedge \psi_{i}(y) \rightarrow \nu_{i}^{2}(x)\wedge \psi_{i}(y)$. Let 
$r_{i} = \mu(\nu_{i}^{1}(x))$ and $t_{i} = \lambda(\psi_{i}(y))$. It follows from the assumptions that
$\sum_{i}r_{i}t_{i}  \leq \omega_{xy}(\phi(x,y)) \leq  \sum_{i}(r_{i}+\epsilon)t_{i} = \sum_{i}r_{i}t_{i} + \epsilon$. 
\newline
Hence $\omega_{xy}(\phi(x,y))$ is determined. 
\end{proof}

\noindent
Note in particular that a smooth measure $\mu_{x}$ has a unique amalgam with any complete type. Note also that if $\mu_{x}$ is a Borel definable global measure, and $\lambda_{y}$ arbitrary then  $\mu_{x}\otimes \lambda_{y}$ is a separated amalgam. It follows that if $\mu_{x}$ is smooth and $\lambda_{y}$ is Borel definable (in the $NIP$ case, invariant) then $\mu_{x}\otimes \lambda_{y} = \lambda_{y}\otimes \mu_{x}$. Namely a smooth measure ``commutes" with any  other invariant measure.

\begin{Corollary} Suppose $\mu_{x}$ is smooth over $M$. Let $\phi(x,y)\in L$
and $X_{1},..,X_{k}$ a finite collection of Borel over $M$ sets . Then for any $\epsilon >0$, for all sufficiently large $m$, there is a formula $\theta_m(x_1,..,x_m)$ such that $lim_{m\rightarrow +\infty}\mu^{(m)}(\theta_m)=1$ and for any $(a_1,..,a_m) \models \theta_m$,
\newline
(i) for each $b\in {\bar M}$, $\mu(\phi(x,b))$ is within $\epsilon$ of $Fr(\phi(x,b), a_{1},..,a_{m})$.
\newline
Also, we can find such $(a_1,..,a_m)$ such that furthermore:
\newline
(ii) $\mu(X_{i})$ is within $\epsilon$ of $Fr(X_{i},a_{1},..,a_{m})$ for each $i$, and
\newline
(iii) for each $b\in {\bar M}$, $\mu(X_{i}\cap\phi(x,b))$ is within $\epsilon$ of 
$Fr(X_{i}\cap\phi(x,b),a_{1},..,a_{m})$.
\end{Corollary}
\begin{proof} Note first that we can assume $x=x$ is an instance of $\phi(x,y)$ and that $x=x$ is included among the $X_i$. Hence (iii) implies (i) and (ii). 

Let $\nu_{i}^{1}(x), \nu_{i}^{2}(x)$ (over $M$) for $i=1,..,n$ say be given by Lemma 2.3 for $\phi(x,y)$ and $\epsilon/4$. Consider the formula $\theta_m(x_1,..,x_m)$ that expresses that $Fr(\nu_i^j(x),x_1,..,x_m)$ is within $\epsilon/4$ of $\mu(\nu_i^j(x))$ for each $i$ and $j$.
The weak law of large numbers, Fact 1.8, applied to $X = S_{x}(M)$, $\mu_{x}|M$ (as a probability measure on $X$), to the Borel sets $\nu_{i}^{j}(x)$ (all $i,j$), and $\epsilon/4$ implies that $lim_{m\rightarrow +\infty}\mu^{(m)}(\theta_m(x_1,..,x_m)) = 1$. (We here use the fact that $\mu_{(x_1,..,x_m)}^{(m)}$ is a separated amalgam of $\mu_{x_1},..,\mu_{x_m}$.)

Next apply the weak law of large numbers, this time to the Borels $\nu_i^j \cap X_r$ (all $i,j,r$) and $\epsilon/4$ to obtain suitable types $p_{1},...,p_{m}\in S_{x}(M)$. Let $a_{1},...,a_{m}$ be realizations of $p_{1},..,p_{m}$ respectively.
Let $\lambda_{x}$ be the average of the $tp(a_{i}/{\bar M})$. Let $b\in {\bar M}$, and $i$ be such that
$\nu_{i}^{1}(x) \rightarrow \phi(x,b) \rightarrow \nu_{i}^{2}(x)$. Then also
for each $r$, 
$\nu_{i}^{1}(x)\wedge x\in X_{r} \rightarrow \phi(x,b)\wedge x\in X_{r} \rightarrow \nu_{i}^{2}(x)\wedge x\in X_{r}$. Also 
clearly $|(\mu(\nu_{i}^{2}(x)\cap X_{r}) - \mu(\nu_{i}^{1}(x)\cap X_{r})|<\epsilon/4$. 

 Now $\lambda(\nu_{i}^{j}(x)\cap X_{r})$ is within $\epsilon/4$ of 
$\mu(\nu_{i}^{j}(x)\cap X_{r})$
for $j=1,2$ from which it follows that  $\lambda(\phi(x,b)\cap X_{r})$ is within $\epsilon$ of $\mu(\phi(x,b)\cap X_{r})$ giving (iii) (so also (i) and (ii), for (i); note that only the fact that $(a_1,..,a_m)\models \theta_m$ is needed).
\end{proof}

\begin{Definition} We will call a global Borel definable Keisler measure $\mu_{x}$ $fim$  (a ``frequency interpretation measure") if: for every $\phi(x,y)\in L$, and $\epsilon>0$, for arbitrary sufficiently large $m$, there is $\theta_m(x_1,..,x_m)$ (with parameters) such that :
\newline
(i) $lim_{m\rightarrow +\infty} \mu^{(m)}(\theta_m) = 1$,
\newline(ii) for all $(a_1,..,a_m) \models \theta_m(x_1,..,x_m)$, $\mu(\phi(x,b))$ is within $\epsilon$ of
$Fr(\phi(x,b),a_1,..,a_m)$. 
\end{Definition}

So we have seen that smooth measures are $fim$.

Note that it follows, as in Corollary 2.6, that for any $fim$ measure, any $\phi(x,y)$ and any Borel set $X$,  we can find $(a_1,..,a_m)$ such that (i) and (ii) of Corollary 2.6 hold. We will see later on that we can also have (iii).

Corollary 2.6   plus Lemma 2.2 enables us to directly prove something about arbitrary measures, for which in \cite{NIPII} we used the Vapnik-Chervonenkis Theorem.

\begin{Corollary} (Assume $T$ has $NIP$.) Let $\mu_{x}$ be any measure over $M$. Let $\phi(x,y)\in L$, $\epsilon > 0$, and let $X_{1},..,X_{k}$ be Borel sets over $M$. THEN for all large enough $n$ there are $a_{1},..,a_{n}$ such that  
for all $r=1,..k$ and all $b\in M$, $\mu(X_{r}\cap \phi(x,b))$ is within $\epsilon$ of
$Fr(X_{r}\cap\phi(x,b),a_{1},..,a_{n})$.
\end{Corollary}
\begin{proof} By Lemma 2.2, let $\mu'$ be global extension of $\mu$ which is smooth over some $M'>M$. Apply Corollary 2.6  to  $\mu'$ and $M'$. 
\end{proof}

We will often use the following consequence of this corollary (in a $NIP$ context): If $\mu$, $\lambda$ are two global invariant measures, assume that $\mu$ commutes with every type $p$ weakly random for $\lambda$, then $\mu$ and $\lambda$ commute. (Here we call $p$ weakly random for $\mu$ if every formula in $p$ has positive $\mu$-measure.)

To see this, given a formula $\phi(x,y)$, write $(\mu(x)\otimes\lambda(y))(\phi(x,y))=\int f(y)d\lambda_y$ as in the paragraph before Lemma 1.5, and approximate that integral by some finite sum $F_N=\sum_{j=1,..,N} \lambda(Y_j)c_j$. Use the corollary to find types $p_1,..,p_n$ weakly random for $\lambda$ such that if $a_i \models p_i$ for all $i$, $Fr(Y_j\cap\phi(x,b),a_1,..,a_n)$ is within $\epsilon$ of $\lambda(Y_j\cap \phi(x,b))$ for all $j$ and $b\in \bar M$.

If $\tilde \lambda$ denotes the average of the types $p_1,..,p_n$ then we leave it to the reader to check that $(\mu(x)\otimes\tilde\lambda(y))(\phi(x,y))$ is close to $(\mu(x)\otimes\lambda(y))(\phi(x,y))$ and $(\tilde \lambda(y)\otimes\mu(x))(\phi(x,y))$ is close to $(\lambda(y)\otimes \mu(x))(\phi(x,y))$. This is enough.

\vspace{2mm}
\noindent
We now begin to discuss ``indiscernibles" in the Keisler measure context.
\begin{Definition} Let $\mu_{x_{1},x_{2},...}$ be a Keisler measure over $M$, where $x_{1}, x_{2},...$ are distinct  variables of the same sort. 
\newline
(i) We say that $\mu_{(x_{i})_{i}}$ is indiscernible if for every formula $\phi(x_{1},..,x_{n})$ over $M$ and all $i_{1} < i_{2} < .. < i_{n}$, 
$\mu(\phi(x_{1},..,x_{n})) = \mu(\phi(x_{i_{1}},..,x_{i_{n}}))$. 
\newline 
(ii) We say that $\mu$ is totally indiscernible if $\mu(\phi(x_{1},..,x_{n})) = \mu(\phi(x_{i_{1}},..,x_{i_{n}}))$, whenever $i_{1},..,i_{n}$ are distinct.
\end{Definition}

So indiscernibility of a measure $\mu$ is with respect to a given sequence $(x_{i})_{i}$ of variables. Likewise we can speak of 
an indiscernible measure in variables $(x_{i}:i\in I)$ where $I$ is a totally ordered index set. We can use compactness to ``stretch" indiscernible measures in the
obvious manner.

\vspace{2mm}
\noindent
We do not know any elementary proof of the following lemma, so we refer the reader to \cite{BY1}. The lemma is an equivalent formulation of Theorem 5.3 of that paper, the equivalence being a consequence of Lemma 5.4 there.

\begin{Lemma} (Assume $T$ has $NIP$.) If $\mu_{(x_{i}:i<\omega)}$ over $M$ is indiscernible,  $\nu(y)$ is a measure over $M$, and $\omega$ is an amalgam of these over $M$, and $\phi(x,y)\in L$, then 
$lim_{i\rightarrow\infty}\omega(\phi(x_{i},y))$ exists. Equivalently, for any such $\phi$, $\mu$, $\nu$, and $\omega$, and $\epsilon$, it is not the case
that $|\omega(\phi(x_{i},y)) - \omega(\phi(x_{i+1},y))| > \epsilon$ for all $i$.
\end{Lemma}

\begin{Remark} ($NIP$) In particular, given indiscernible measure $\mu_{(x_{i})_{i}}$ and an extension $\mu'$ over $M'$ then we obtain a unique Keisler measure in single 
variable $x$ over $M'$, which we call  $Av(\mu',M')$, whose measure of $\phi(x,c)$  (for $c\in M'$) is $lim_{i\to\infty}(\mu'(\phi(x_{i},c)))$.
\end{Remark}

\begin{Corollary} ($NIP$)  For any formula $\phi(x,y)$ and $\epsilon$ there is $N$ such that for any indiscernible Keisler measure
$\mu_{(x_{i}:i<\omega)}$ over a model $M$ (or set $A$), $b\in {\bar M}$ and extension $\mu'$ of $\mu$ over $(M,b)$, there do not exist
$i_{1} < i_{2} < ... < i_{N}$ such that $|\mu'(\phi(x_{i_{j}},b)) - \mu'(\phi(x_{i_{j+1}},b))| > \epsilon$ for all $j = 1,..,N-1$.
\end{Corollary}
\begin{proof} By compactness, in the space of measures over $M$ in variables $((x_{i})_{i<\omega},y)$.
\end{proof}

\begin{Corollary} ($NIP$) Given $\phi(x,y)$ and $\epsilon> 0$ there is $N$ such that for any totally indiscernible Keisler measure $\mu_{(x_{i}:i<\omega)}$ over a model $M$, and $b\in {\bar M}$ and extension $\mu'$ of $\mu$ over $(M,b)$, for any $r\in [0,1]$ either
$\{i: \mu'(\phi(x_{i},b)) \geq r+\epsilon\}$ has cardinality $< N$, or $\{i: \mu'(\phi(x_{i},b) \leq r-\epsilon\}$ has cardinality $< N$.
\end{Corollary}
\begin{proof} Let $N$ be as given by Corollary 2.12 for $\phi(x,y)$ and $\epsilon$. Then clearly it also works for the present result. 
\end{proof}

\begin{Lemma} ($NIP$.)  Let $\mu$ be a global invariant Keisler measure. Then
\newline
(i) $\mu^{(\omega)}$ is indiscernible,
\newline
(ii) if both $\mu$ and $\nu$ are $A$-invariant and $\mu^{(\omega)}|A = \nu^{(\omega)}|A$ then $\mu = \nu$.
\end{Lemma}
\begin{proof}
(i) Obvious and easily proved by induction.
\newline
(ii)  This is as in the type case: namely suppose for a contradiction that 
$\mu(\phi(x,b))= r \neq s = \nu(\phi(x,b))$. Let $\lambda_{1}(x_{1}) = \mu(x_{1})$, and for even $n$, $\lambda_{n}(x_{1},..,x_{n}) = \nu(x_{n})\otimes \lambda_{n-1}(x_{1},..,x_{n-1})$ and for odd $n$, 
$\lambda_{n} = \mu(x_{n})\otimes \lambda_{n-1}(x_{1},..,x_{n-1})$.
Let $\lambda_{x_{1},x_{2},...}$ be the union. Then one checks that $\lambda|A =
\mu^{(\omega)}|A = \nu^{(\omega)}|A$. But $\lambda(\phi(x_{i},b)) = r$ for odd $i$ and equals $s$ for even $i$, contradicting Lemma 2.10.
\end{proof}

\section{Generically stable measures}
This section contains our main results. Namely Theorem 3.2 below which gives equivalent conditions for a measure to be generically stable in an $NIP$ theory. We assume $NIP$ throughout, although it would not be uninteresting to develop the theory for arbitrary $T$. We give two proofs of that theorem, the first one follows very closely the proof of the analogous result for types (Proposition 3.2  of \cite{NIPII}) while the second one is inspired by the proof of the Vapnik-Chervonenkis theorem and avoids the use of Lemma 2.10.

We begin by giving a promised generalization of Lemma 3.4 of \cite{NIPII}  to measures.

 \begin{Lemma} Suppose that $\mu_{x}$ and $\lambda_{y}$ are global Keisler measures such that $\mu$ is finitely satisfiable (in some small model) and $\lambda$ is definable. Then $\mu_{x}\otimes \lambda_{y} = \lambda_{y}\otimes \mu_{x}$.
 \end{Lemma}
 \begin{proof} We first note that using the remark following 2.8 it suffices to prove the lemma when $\mu$ is a type $p$ say. 
 So assume $M$ is a small model over which the TYPE $p(x)$ is finitely satisfiable and the measure $\lambda_{y}$ is definable.
 Let $\phi(x,y)\in L$. Suppose for a contradiction that $\phi(x,y)\in L$ and $(p(x)\otimes \lambda_{y})(\phi(x,y)) = r$, 
$(\lambda_{y}\otimes p(x))(\phi(x,y)) = s$, and $r\neq s$. Let $\epsilon = |r-s|/4$.
Note that $s$ is precisely $\lambda(\phi(a,y))$ where $a$ is some (any) realization of $p|M$. By definability of $\lambda$ over $M$, let $\theta(x)$ be a formula over $M$, which is in $p|M$, and such that 
\newline
(I) if $\models\theta(a')$ then $\lambda(\phi(a',y))$ is strictly
within $\epsilon$ of $s$. 

By Borel definability of $p$ over $M$, $X = \{b: \phi(x,b)\in p\}$ is Borel over $M$. Hence $r = \lambda(X)$. Apply 2.8 to $\lambda|M$ to find $b_{1},..,b_{n}\in {\bar M}$ such that 
\newline
(II) for all $a'\in M$, $\lambda(\phi(a',y))$ is within $\epsilon$ of $Fr(\phi(a',y),b_{1},..,b_{n})$ and
\newline 
(III) the proportion of $b_{i}$'s in $X$ is within $\epsilon$ of $r$. Suppose for simplicity that that $b_{i}\in X$ for precisely $i = 1,..,m$.

Let $a$ realize $p|(M,b_{1},...,b_{n})$. Hence, by the definition of $X$, $\models \phi(a,b_{i})$ just if $1\leq i \leq m$.
Also of course $\models \theta(a)$.  By finite satisfiability of $p$ in $M$ there is $a'\in M$ such that $\models\theta(a')$,
$\models \phi(a',b_{i})$ for $i=1,..,m$, and $\models \neg\phi(a',b_{i})$ for $i=m+1,..,n$. 
By (II) $\lambda(\phi(a',y))$ is within $\epsilon$ of $m/n$. On the other hand by (I) $\lambda(\phi(a',y))$ is within $\epsilon$ of $s$.
As $m/n$ is within $\epsilon$ of $r$. we have a contradiction.
 \end{proof}

 \vspace{2mm}
 \noindent
 \begin{Theorem} Suppose that $\mu(x)$ is a global Keisler measure which is $A$-invariant. Then the following are equivalent:
 \newline
 (i) $\mu$ is both definable (necessarily over $A$) and finitely satisfiable in a small model (necessarily in any model containing $A$),
 \newline
 (ii) $\mu^{(\omega)}_{(x_{1},x_{2}...)}|A$ is totally indiscernible,
 \newline
 (iii) $\mu$ is $fim$,
 \newline
 (iv) for any global $A$-invariant Keisler measure $\lambda_y$, $\mu_{x}\otimes \lambda_{y} =
 \lambda_{y}\otimes \mu_{x}$,
 \newline
 (v) $\mu$ commutes with itself : $\mu_x \otimes \mu_y = \mu_y \otimes \mu_x$.
 \newline
 (vi)  for some small model $M_{0}$ containing $A$, for any Borel over $A$ set $X$ and any formula $\phi(x)$ over $\bar M$,  if $\mu(X\cap\phi(x)) > 0$ then there is $a\in M_{0}$ such that $a\in X$ and $\phi(a)$.  
 
  \end{Theorem}
 \begin{proof} (i) implies (ii): 
 Note that if $\mu$ and $\lambda$ are both definable measures, then so is $\mu\otimes \lambda$. So we see that each $\mu^{(n)}_{x_{1},..,x_{n}}$ is definable. By Lemma 3.1
 $\mu^{(n+1)}_{x_{1},..,x_{n+1}} =_{def}\mu_{x_{n+1}}\otimes \mu^{(n)}_{x_{1},..,x_{n}} = 
 \mu^{(n)}_{x_{1},..,x_{n}}\otimes \mu_{x_{n+1}}$. It follows easily (using indiscernibility of each $\mu^{(n)}$) that each $\mu^{(n)}_{x_{1},..,x_{(n)}}$ is totally indiscernible, hence so is $\mu^{(\omega)}$.
 
 \vspace{2mm}
 \noindent
 (iii) implies (i): For all $\phi(x,y)\in L$, $\epsilon = 1/n$ and sufficiently large $m$, the definition of $fim$ supplies us with a formula $\theta(x_1,..,x_{m})$. Take a model $M$ containing the parameters of those formulas for all $\phi$, $n$ and $m$. Then $\mu$ is definable and finitely satisfiable over $M$.
 
 \vspace{2mm}
 \noindent
 (ii) implies (iii):  Without loss $A = M$ is a small model. Assume $\mu^{(\omega)}$ totally indiscernible.
 \newline
 {\em Claim.}  Suppose $\lambda_{x_{1},x_{2},..}$ is an extension of $\mu^{(\omega)}_{x_{1},x_{2},..}|M$ to a model $M'>M$. Then $Av(\lambda,M')$ is precisely $\mu|M'$. 
 \newline
 {\em Proof of Claim.} Otherwise, we have some formula $\phi(x,c)$ over $M'$ such that
 $\mu(\phi(x,c)) = r$ say, and $Av(\lambda,M')(\phi(x,c)) = s \neq r$. Without loss $s > r$, and let $\epsilon = s-r$. Let $N_{\phi,\epsilon}$  be given by Corollary 2.13. 
By Lemma 2.10, $\lambda(\phi(x_{i},c))> s-\epsilon$ for eventually all $i$.
However let $\alpha_{(x_{i}:i<\omega + \omega)}$ be 
$\mu^{(\omega)}_{(x_{i}:\omega \leq i < \omega+\omega)}|M'\otimes \lambda_{(x_{i}:i<\omega)}$, a measure over $M'$. Then $\alpha$ is clearly an extension of $\mu^{(\omega+\omega)}$ to $M'$. But  $\alpha(\phi(x_{i},c)) = r$ for all $i\geq\omega$, and we have a contradiction to the existence of $N_{\phi,\epsilon}$. The claim is proved.

\vspace{2mm}
\noindent
Let $\epsilon>0$ and $\phi(x,y) \in L$. Let $N$ be given by Corollary 2.13 for $\phi$ and $\epsilon$ and let $M=4N$. Then, by the Claim, for any two measures $\lambda,\lambda'$ extending $\mu^{(\omega)}$, for any $b\in \bar M$, we have $|\{i:|\lambda(\phi(x_i,b))-\lambda'(\phi(x_i,b))|\geq 2\epsilon\}|<2N$. By compactness, there is a formula $\Phi(x_1,..,x_M)$ and a small $r>0$ such that for any measure $\nu_{(x_1,..,x_M)}$ such that $|\nu(\Phi)-\mu^{(M)}(\Phi)|\leq r$, we have $|\{i:|\nu(\phi(x_i,b))-\mu^{(M)}(\phi(x_i,b))|\geq 2\epsilon\}|<2N$.

For sufficiently large $k$, let $\theta_{kM}(x_1,..,x_{kM})$ be the formula that expresses that $Fr(\Phi(x_1,..,x_M);y_0,..,y_{k-1})$ is within $r/2$ of $\mu_{(M)}(\Phi)$, where $y_i$ denotes the tuple of variables $(x_{Mi+1},..,x_{Mi+M})$. If $kM < n < (k+1)M$, define $\theta_n(x_1,..,x_n) = \theta_{kM}(x_1,..,x_{kM})$. Then by the weak law of large numbers, $lim_{n\rightarrow +\infty}\mu^{(\omega)}(\theta_n)=1$. Futhermore, for sufficiently large $n$, if $(a_1,..,a_n)\models \theta_n$ then, letting $\nu$ be the average of $tp(a_1/\bar M),..,tp(a_n/\bar M)$, $\nu(\phi(x,b))$ is within $3\epsilon$ of $\mu(\phi(x,b))$ for every $b\in \bar M$. This shows that $\mu$ is $fim$.

\vspace{2mm}
\noindent
(iii) implies (iv).  We have to prove that any $fim$ measure commutes with any invariant measure. As above it suffices to prove that an $fim$ measure $\mu_{x}$ commutes with any invariant type.

 Let $q(y)$ be such. Assume both $\mu$ and $q$ are $M$-invariant. Let $\phi(x,y)\in L$.  Note that $(\mu_{x}\otimes q(y))(\phi(x,y)) = \mu(\phi(x,b)) = r$ for some (any) $b$ realizing $q|M$. And also
$(q(y)\otimes\mu_{x})(\phi(x,y)) = \mu(X) = s$ where $X = \{a:\phi(a,y)) \in q(y)\}$ (a Borel set over $M$). For given $\epsilon$ choose a set $a_{1},....,a_{k}$ witnessing $fim$ for $\mu$ with respect to $\phi(x,y)$ and such that $Fr(X;a_1,..,a_k)$ is within $\epsilon$ of $\mu(X)$. Let $b$ realize $q|(M,a_{1},..,a_{k})$. So $\mu(X)$ is within $\epsilon$ of $Fr(\phi(x,b);a_1,..,a_k)$. But the latter is within $\epsilon$ of $\mu(\phi(x,b))$. So $r = s$ and we are finished. 

\vspace{2mm}
\noindent
(iv) implies (v). Obvious. 
 
 \vspace{2mm}
 \noindent
(v) implies (ii). This follows from associativity of $\otimes$: for any $k <n$ we have assuming (v), $\mu_{x_1} \otimes ... \otimes \mu_{x_k}\otimes \mu_{x_{k+1}} \otimes ... \otimes \mu_{x_n} = \mu_{x_1} \otimes ... \otimes \mu_{x_{k+1}}\otimes \mu_{x_{k}} \otimes ... \otimes \mu_{x_n}$. This is enough.

\vspace{2mm}
\noindent  (vi) is a form of  "Borel satisfiability".  It is analogous to (i) since Borel definability is automatic.  
The proof of equivalence with the other conditions  is postponed to Lemma 3.6 and Theorem 4.8 below.

 \end{proof}
 
 \vspace{2mm}
 \noindent
 We call a global Keisler measure $\mu$ {\em generically stable} if it satisfies the equivalent conditions of Theorem  3.2. (Of course assuming $T$ is $NIP$.)

 \begin{Proposition} Suppose that $\mu$ is generically stable and $A$-invariant. Then $\mu$ is the unique $A$-invariant extension of $\mu|A$.
 \end{Proposition}
 \begin{proof}  Suppose that $\nu$ is $A$-invariant and $\nu|A = \mu|A$. 
 By property (iv) above we check inductively that $\mu^{(n)}|A = \nu^{(n)}|A$ for all $n$. By Lemma 2.14, $\mu = \nu$.
 \end{proof}
 
 We give now a proof of Theorem 3.2 which does not use Lemma 2.10. That lemma was used only in the implication $(ii)\rightarrow (iii)$. So we give an alternative proof that if an invariant measure $\mu$ is such that $\mu^{(\omega)}$ is totally indiscernible, then $\mu$ is $fim$.
\\

By a symmetric measure on some $X^{n}$, $X$ a definable set, we will now mean a measure $\mu_{(x_1,..,x_n)}$ such that $\mu(x_i\in X)=1$ for all $i$ and for any $\sigma\in S_n$ and formula $\phi(x_1,..,x_n)$, $\mu(\phi(x_1,..,x_n))=\mu(\phi(x_{\sigma.1},..,x_{\sigma.n}))$.

The following crucial lemma is related to the classical Vapnik-Chervonenkis theorem (see \cite{Vapnik-Chervonenkis}) and could be proved by similar methods. But it does not seem to be a direct consequence of it.

\begin{Lemma} (Assume $T$ has $NIP$.)
Let $\phi(x,y)\subseteq X\times Y$ be a formula over a model $M$. For $n>0$, let $\mu_n$ be any symmetric, $M$-invariant global measure on $X^{2n}$. Given $b\in Y$ and $a=(a_1,..,a_n) \in X^{n}$, let $f(a;b)=Fr(\phi(x,b);a_1,..,a_n)$. Let $$\delta_0(a,a';b) = |f(a;b)-f(a';b)|,$$
$$\delta(a,a')=\sup_{b\in \bar M} \delta_0(a,a';b).$$
Finally, let $E(n)$ be the $\mu_n$-expectation of $\delta$. Then $\lim_{n\rightarrow \infty}E(n)=0$.
\end{Lemma}

\begin{proof}
Note first that $\delta$ is measurable for the boolean algebra generated by the definable subsets of $X^{2n}$ of the form: $(\exists y)(\bigwedge_{i=1}^{2n} \phi(x_i,y)^{\nu(i)})$ (where $\phi(x_i,y)^{\nu(i)}$ is either $\phi(x_i,y)$ or $\neg\phi(x_i,y)$). In particular it makes sense to ask for the $\mu_n$-expectation of $\delta$.

 Fix $\epsilon>0$, and let $n$ be large compared
to $\epsilon$.   

Let $\Zz/2$ act on the variables $\{x_i,x_i'\}$ by flipping them, and let $(\Zz/2)^n$ act 
on $\{x_1,\ldots,x_n,x_1',\ldots,x_n' \}$ by the product action.   

Given $(a,a')$ and $b$, we have  $|\delta_0(a,a';b)| \leq 1$.
  Let  $X(b)= \{s \in (\Zz/2)^n: \delta_0(s(a,a');b) >  {\epsilon} \}$ and let $c(b) = \{i: \phi(a_i,b) \not \equiv \phi(a'_i,b) \}$.   If $c(b)=\emptyset$ 
then $\delta_0(s(a,a');b)=0$ for all $s$.  Otherwise, we   view $\delta(s(a,a');b)$ as a random variable of $s $ (on a finite probability space). More precisely, it is the absolute value of a sum of $|c(b)|$ independent random variables each of expectation 0 and variance $1/n^2$. Therefore $\delta_0(s(a,a');b)$ has expectation $0$, and variance $|c(b)|/n^2$.  By
Tchebychev's inequality we have $|X(b)| \leq |c(b)|/(n \epsilon)^2 \leq \epsilon^{-2}$.
So $|X(b)| / 2^n \leq \epsilon^{-2} 2^{-n}$.

Now $X(b)$
depends only on $\{i: \phi(a_i ,b) \}$ and $\{i: \phi(a'_i,b) \}$; there are polynomially
many possibilities for these sets, by $NIP$.  Hence, $\union_b X(b)$ is an exponentially small subset of $(\Zz/2)^n$.  If $n$ is large enough, it has   proportion $< {\epsilon}$.  Let $s \notin \union_b X(b)$.  Then $|\delta_0(s(a,a'),b)| \leq {\epsilon} $ for all $b$.  If $s \in \union_b X(b)$ we have
at any rate $|\delta_0(s(a,a');b)| \leq 1$.  Thus  
 $2^{-n} \sum_s \sup_b \delta(s(a,a');b) < 2{\epsilon}$.

 By the symmetry of $\mu_n$, $E(n)$ equals the $\mu_n$-expectation of $\sup_a \delta(s(a,a'))$  for any $s   \in (\Zz/2)^n$, hence it is also equal to the average $2^{-n} \sum_s \sup_a \delta(s(a,a'))  $.  So $E(n) < 2{\epsilon}$.  

\end{proof}

\begin{Corollary} (NIP) Let $\mu$ be an $M$-invariant global measure, such that $\mu^{(\omega)}$ is totally indiscernible, then $\mu$ is $fim$.
 \end{Corollary}

\begin{proof}   
Let $\phi(x,y)$ be a formula, and take $\epsilon >0$. By the previous lemma, for large enough $n$, the set $W=\{(a,a'):\delta(a,a')<\epsilon/4\}$ satisfies $\mu^{(2n)}(W)\geq 1-\epsilon$. Note that this is a definable set. 
Therefore there exists  $a$ such that $\mu^{(n)}(W(a)) \geq 1- 2 \epsilon > 1/2$. (Where $W(a)=\{a':(a,a')\in W\}$.) 
Now fix $a \in \bar M$.  Then for all $b'$, $\delta_0(a,a';b) \leq \delta(a,a')$.

On the other hand
let $Q_n'(b)$ be the set of $a'$ such that $| f(a';b) - \mu(\phi(x,b))| \geq  \epsilon/2$.  
By Tchebychev's inequality, and since the variance of the truth value of $\phi(x,b)$
is at most $1$, we have $\mu^{(n)}(Q_n'(b)) \leq 1 / ( n (\epsilon/2)^2)$.  Let $Q_n(b)$ be the complement of 
$Q_n'(b)$,  and assume $n> 2(\epsilon/2)^{-2}$ (note that this does not depend on $b$).  Then $\mu^{(n)}(Q_n(b))  > 1/2$.   Hence  there exists $a' \in W(a) \cap Q_n(b)$.  So 
$\delta_0(a,a';b) < \epsilon/4$ and $| f(a';b) -\mu(\phi(x,b))| <  \epsilon/2$. Now for any $a'' \in W(a)$, we have $\delta_0(a',a'';b)< \epsilon/2$, therefore $| f(a'',b) - \mu(\phi(x,b)) | < \epsilon$.

So we have found, for large enough $n$ a formula $\theta'_n=W(a)$ satisfying condition (ii) in the definition of $fim$ with $\mu^{(n)}(\theta'_n)\geq 1-2\epsilon$. As this is true for all $\epsilon$, we can construct a sequence of formulas $\theta_n(x_1,..,x_n)$ satisfying the same condition, but with $\mu^{(n)}(\theta_n)\rightarrow 1$. This proves that $\mu$ is $fim$.
\end{proof}

\begin{Lemma}\label{fim2}
Let $\mu$ be an $M$-invariant global $fim$ measure. Let $\phi(x,y)$ be a formula over $M$ and let $X$ be a Borel over $M$ set. Then for any $\epsilon >0$, for some $m$, we can find $(a_1,..,a_m)$ such that for each $b \in \bar M$, $\mu(X \cap \phi(x,b))$ is within $\epsilon$ of $Fr(X \cap \phi(x,b),a_1,..,a_n)$.
\end{Lemma}
\begin{proof}
We know that $\mu^{(\omega)}$ is totally indiscernible.
The proof is then a slight modification of the lemma above and its corollary. First, in the lemma, change the definition of $f$ to $f(a;b)=Fr(\phi(x,b)\wedge x\in X;a_1,..,a_n)$. Define $\delta_0$ and $\delta$ accordingly. Then the proof goes through without any difficulties. The corollary also goes through with the new definitions of $f$, $\delta_0$ and $\delta$, only $W$ and $W(a)$ are no longer definable. Still, $W(a)$ is a Borel set of measure greater than $1/2$ and for any $a' \in W(a)$, and any $b$, we have $|f(a',b)-(\mu(\phi(x,b)\cap X))|<\epsilon$.  
\end{proof}

 \vspace{5mm}
 \noindent
 Finally we will point out how generically stable measures are very widespread
 in $NIP$ theories, in fact can be constructed from any indiscernible sequence. 
By an indiscernible {\em segment} we mean $(a_{i}:i\in [0,1])$ which is indiscernible with respect to the ordering on the real unit interval $[0,1]$. For any formula $\phi(x,b)$, $\{i\in [0,1]:\models \phi(a_{i},b)\}$ is a finite union of convex sets, and hence intervals. (See \cite{Adler} for example.) We define a measure $\mu_{x}$ as follows: $\mu(\phi(x,b))$ is the Lebesgue measure of $\{i\in [0,1]:\models \phi(a_{i},b)\}$ (i.e. just the sum of the lengths of the relevant disjoint intervals).
 Clearly $\mu_{x}$ is a global Keisler measure on the sort of $x$. 
 
 \begin{Proposition} The global Keisler measure $\mu_{x}$ constructed above is generically stable.
 \end{Proposition}
 \begin{proof} Let $A = \{a_{i}:i\in [0,1]\}$. We show that $\mu$ is both finitely satisfiable in $A$ and definable over $A$. Finite satisfiability is clear from the definition of $\mu$. (If $\mu(\phi(x,b)) > 0$ then the Lebesgue measure of $C = \{i:\models\phi(a_{i},b)\}$ is $>0$ hence $C\neq \emptyset$.)

To show definability, we in fact note that $\mu$ is $fim$. First note that for any formula $\phi(x,y)$, there is $N_{\phi}$ such that for all $b$ $\{i\in [0,1]:\models\phi(x,b)\}$ is a union of at most $N$ disjoint intervals. Hence, given $\epsilon>0$, if we choose $0< \delta < \epsilon/N$, and let
$i_{k} =  0 + k\delta$ for $k$ such that $k\delta \leq 1$, then for any $b$, $\mu(\phi(x,b))$ is within $\epsilon$ of the proportion of $i_{k}$ such that $\models\phi(a_{i_{k}},b)$. 
 \end{proof}

\section{The $fsg$ property for groups, types and measures}  
We will again make a blanket assumption that $T$ has $NIP$ (but it is not always needed). 
In \cite{NIPI} and \cite{NIPII}  definable groups $G$ with finitely satisfiable generics ($fsg$) played an important  role. This $fsg$ property asserted the existence of a global type of $G$ every left translate of which was finitely satisfiable in some given small model. By definition this is a property of a definable group, rather than of some global type or measure. We wanted to find adequate generalizations of the $fsg$ notion to arbitrary complete types $p(x)$ over small sets, and even arbitrary Keisler measures over small sets. A tentative definition of a complete type $p(x)\in S(A)$ having $fsg$ was given in \cite{NIPII}. We try to complete the picture here, making the connection with generically stable global measures. We should say that the subtlety
of the $fsg$ notion is really present in the case where the set $A$ is NOT bounded closed. In the group case this corresponds to the case where $G\neq G^{00}$. 

We first return to the group case, adding to results from \cite{NIPII}. 

Recall the original definition:
\begin{Definition} The definable group $G$ has $fsg$ if there is a global complete type $p(x)$ of $G$ and a small model $M_{0}$ such that every left translate of $p$ is finitely satisfiable in $M_{0}$.
\end{Definition}

As pointed out in \cite{NIPI}  $M_{0}$ can be chosen as any model over which $G$ is defined. 
In \cite{NIPII} the notion was generalized to type-definable groups $G$.

\begin{Remark} The definable group $G$ has $fsg$ if and only if there is global left invariant Keisler measure $\mu$ on $G$ which is finitely satisfiable in some (any) small model $M_{0}$ over which $G$ is defined.
\end{Remark}
\begin{proof}
Assuming that $G$ has $fsg$, Proposition 6.2 of \cite{NIPI}  produces the required measure $\mu$. Conversely, supposing that $\mu$ is a global left invariant Keisler measure on $G$, finitely satisfiable in $M_{0}$, let $p$ be some global type of $G$ such that $\mu(\phi) > 0$ for all $\phi\in p$. Namely $p$ comes from an ultrafilter on the Boolean algebra of definable subsets of $G$ modulo the equivalence relation $X\sim Y$ if $\mu(X\triangle Y) = 0$. Then every left translate of every formula in $p$ has $\mu$ measure $> 0$ so is realized in $M_{0}$. 
\end{proof}

The following strengthens the  ``existence and uniqueness" of (left/right) $G$-invariant global Keisler measures for $fsg$ groups $G$, from \cite{NIPII}. But the proof is somewhat simpler, given the results established in earlier sections.

\begin{Theorem} Suppose $G$ has $fsg$, witnessed by a global left invariant Keisler measure $\mu$ finitely satisfiable in small $M_{0}$. 
Then 
\newline
(i) $\mu$ is the unique left invariant global Keisler measure on $G$, as well as the unique right invariant Keisler measure on $G$.
\newline
(ii) $\mu$ is both the unique left $G^{00}$-invariant global Keisler measure, as well as the unique right $G^{00}$-invariant global Keisler measure on $G$, which extends Haar measure $h$ on $G/G^{00}$.
\end{Theorem}
\begin{proof}
(i) This is precisely 7.7 of \cite{NIPII}. But note that we can use Lemma 3.1 of the present paper and the relatively soft 5.8 of \cite{NIPII} in place of 7.3 and 7.6 of \cite{NIPII}. (Details: Suppose $\lambda$ is also global left invariant. By the Lemma 5.8 of \cite{NIPII}, we may assume $\lambda$ to be definable. Given definable subset $D$ of $G$, let 
$Z= \{(g,h):g\in hZ\}$. By 3.1, $\mu_{x}\otimes \lambda_{y} = \lambda_{y}\otimes \mu_{x}$, whence  $(\mu_{x}\otimes \lambda_{y})(Z) = \mu(D)$ and
$(\lambda_{y}\otimes\mu_{x})(X) = \lambda(D^{-1})$. So $\lambda = \mu^{-1}$. 
This in particular yields that $\mu = \mu^{-1}$. So $\lambda = \mu$ and $\mu$ is also the unique right invariant Keisler measure.)
\newline
(ii) Note first that $\mu_{x}$ induces a left invariant measure on $G/G^{00}$ which has to be (normalized) Haar measure, and of course $\mu$ is (left/right) $G^{00}$-invariant. Let $\lambda_{y}$ be another global left $G^{00}$-invariant Keisler measure extending (or inducing) Haar measure on $G/G^{00}$. As in 5.8 of \cite{NIPII} we may assume that $\lambda$ is definable. 

Let $X$ be a definable subset of $G$. Let $r = \lambda(X)$. Choose $\epsilon > 0$. As $G^{00}$ stabilizes $\lambda$, and $\lambda$ is definable, there is a definable subset $Y$ of $G$ which contains $G^{00}$ and such that for all $g\in Y$, $\lambda(gX)\in (r-\epsilon,r+\epsilon)$. Let $\pi:G\to G/G^{00}$ be the canonical surjective homomorphism. Then $\{c\in G/G^{00}: \pi^{-1}(c) \subseteq Y\}$ is an open neighbourhood $W$ of the identity in $G/G^{00}$. Let $U = \pi^{-1}(W)\subseteq Y \subseteq G$. Note that $\mu(U) > 0$ (as it equals the Haar measure of the open subset $W$ of $G/G^{00}$).
We have
\newline
(1) for all $g\in U$, $\lambda(gX) \in (r-\epsilon, r+\epsilon)$. 
\newline
Whence
\newline
(2)  $(r-\epsilon)\mu(U) \leq \int_{g\in U}\lambda(gX)d\mu \leq (r+\epsilon)\mu(U)$.
\newline
Now let $Z = \{(h,g):h\in gX$ and $g\in G\}$. Then
\newline
(3) $(\lambda_{y}\otimes \mu_{x})(Z) = \int_{g\in U}\lambda_{y}(gX)d\mu_{x}$. 
\newline
On the other hand clearly
\newline
(4) $(\mu_{x}\otimes\lambda_{y})(Z) = \int_{h\in G}\mu_{x}(hX^{-1} \cap U)d\lambda_{y}$. 
\newline
Note that the value of $\mu_{x}(hX^{-1}\cap U)$ depends only on $h/G^{00}$, hence as $\lambda$ and $\mu$ agree ``on $G/G^{00}$", we see that
\newline
(5)  $(\mu_{x}\otimes \lambda_{y})(Z) = (\mu_{x}\otimes\mu_{y})(Z)$.
\newline
By 3.1 applied twice (to $(\mu_{x}\otimes \lambda_{y})$ AND to 
$(\mu_{x}\otimes\mu_{y})$) together with (5) we see that
\newline
(6) $(\lambda_{y}\otimes \mu_{x})(Z) = (\mu_{y}\otimes\mu_{x})(Z)$. 
\newline
But $(\mu_{y}\otimes\mu_{x})(Z) = \int_{g\in U}\mu_{y}(gX)d\mu_{x} = \mu(X)\mu(U)$. 
\newline
So using (2) and (3) we
 see that
 \newline
 (7)  $(r-\epsilon)\mu(U) \leq \mu(X)\mu(U) \leq (r+\epsilon)\mu(U)$.
 \newline
 So $r-\epsilon \leq \mu(X) \leq r+\epsilon$. As this is true for all $\epsilon$ we conclude that $\mu(X) = r = \lambda(X)$.

\end{proof}

\begin{Remark} The definable group $G$ has $fsg$ if and only if $G$ has a global generically stable left invariant measure.
\end{Remark}

We now consider the general situation. We first recall the definition from \cite{NIPII}:
\begin{Definition} $p(x)\in S(A)$ has $fsg$ if $p$ has a global extension $p'$ such that  for any formula $\phi(x)\in p'$ and $|A|^+$-saturated model $M_{0}$ containing $A$, there is $a\in M_{0}$ realizing $p$ such that $\models \phi(a)$.
\end{Definition}

\begin{Lemma} $p(x)\in S(A)$ has $fsg$ iff there is a global $A$-invariant measure
$\mu_{x}$ extending $p$ such that whenever $\phi(x)$ is a formula over $\bar M$ with $\mu$-measure $> 0$, then for any $|A|^+$-saturated model $M_{0}$ containing $A$, there is $a\in M_{0}$ realizing $p$ such that $\models\phi(a)$. 
\end{Lemma}
\begin{proof}  RHS implies LHS. This is trivial because any weakly random type for the measure $\mu$ will satisfy Definition 4.5.
\newline
LHS implies RHS:  Let $p'$ be as given by Definition 4.5. Then  7.12 (i) of \cite{NIPII} says that $p'$ is a nonforking extension of $p$. Moreover it is clear that any $Aut({\bar M}/A)$-conjugate of $p'$ also satisfies Definition 4.5. Let $\mu$ be the global $A$-invariant measure extending $p$, constructed from $p'$ in Proposition 4.7 of \cite{NIPII}. Then any formula with positive $\mu$-measure must be in some $Aut({\bar M}/A)$-conjugate of $p'$ so is satisfied in any saturated model $M_{0}$ containing $A$ by a realization of $p$. 
\end{proof}

The last lemma motivates a definition of $fsg$ for arbitrary measures over $A$.
\begin{Definition} Let $\mu_{x}$ be a Keisler measure over $A$. We say that $\mu$ has $fsg$, if $\mu$ has a global $A$-invariant extension $\mu'$ such that for any Borel over $A$ set $X$, formula $\phi(x)$ over $\bar M$, and $|A|^+$-saturated model $M_{0}$ containing $A$, if $\mu'(X\cap\phi(x)) > 0$ then there is $a\in M_{0}$ such that $a\in X$ and $\phi(a)$.
\end{Definition}

\begin{Theorem} 
Let $\mu_x$ be a measure over a set $A$, then the following are equivalent :
\newline
(i) $\mu$ has a unique $A$-invariant global extension $\mu'$ that is moreover generically stable.
\newline
(ii) $\mu$ has $fsg$.
\newline
(iii) $\mu$ has a global $A$-invariant extension $\mu'$ such that for any Borel over $A$ set $X$, formula $\phi(x)$ over $\bar M$, and $|A|^+$-saturated model $M_{0}$ containing $A$, if $\mu(X)+\mu(\phi(x)) > 1$ then there is $a\in M_{0}$ such that $a\in X$ and $\phi(a)$.
\end{Theorem}
\begin{proof}  (i) implies (ii):  Follows from Lemma \ref{fim2}.
\newline
(ii) implies (iii) is clear.
\newline
(iii) implies (i): We fix global $A$-invariant measure $\mu'$ extending $\mu$ and witnessing the assumption.
We will prove that $\mu'$ commutes with every $A$-invariant measure. It will follow
that $\mu'^{(\omega)}|A$ is totally indiscernible, so $\mu'$ is generically stable. Uniqueness follows from 3.3.  The proof will be a bit like that of 3.1. In fact it is easy to see that $\mu$ commutes with any $A$-invariant type. If $A = bdd(A)$ this would suffice (as every $bdd(A)$-invariant measure is ``approximated" by $bdd(A)$-invariant types). But for arbitrary $A$ it does not seem to suffice.

So let us fix an $A$-invariant (thus Borel definable over $A$) global measure $\lambda_{y}$. 
\newline
Let $P=(\mu_{x}\otimes\lambda_{y})(\phi(x,y))=\int\mu(\phi(x,b))d\lambda$
and $R=(\lambda_{y}\otimes\mu_{x})(\phi(x,y))=\int\lambda(\phi(a,y))d\mu$.
We want to show that $P=R$.

For any $t\in[0,1]$, let $C_{t}=\{q\in S(A):\mu(\phi(x,b))\geq t\mbox{ for any }b\models q\}$ and $B_t=\{p\in S(A):\lambda(\phi(a,y))\geq t\mbox{ for any }a\models p\}$. These sets are Borel over $A$.

Let $\epsilon>0$ and take $N\geq1/\epsilon$ such that $ $\[
\left|P-\frac{1}{N}\sum_{k=0}^{N-1}\lambda(C_{k/N})\right|\leq\epsilon.\]

Take a model $M$ containing $A$ and $|A|^{+}$-saturated. By Corollary 2.8,
there exist $n$ and $p_{1},..,p_{n}\in S(M)$ such that $Fr(\phi(a;y);p_{1},..,p_{n})$
is within $\epsilon$ of $\lambda(\phi(a;y))$ for every $a\in M$
and $Fr(C_{k/N};p_{1},..,p_{n})$ is within $\epsilon$ of $\lambda(C_{k/N})$
for every $k<N$. Realize $p_{1},..,p_{n}$ in $\bar{M}$ by $b_{1},..,b_{n}$
respectively. Call $\tilde{\lambda}$ the average measure of $b_{1},..,b_{n}$
(seen as global measures).

By construction, we have\[
\left|\frac{1}{N}\sum_{k=0}^{N-1}\lambda(C_{k/N})-\frac{1}{N}\sum_{k=0}^{N-1}\tilde{\lambda}(C_{k/N})\right|\leq\epsilon.\]

On the other hand, for all $k<N$ :\[
\tilde{\lambda}\left(C_{k/N}\right)=\frac{1}{n}\left|\{i:p_{i}\in C_{k/N}\}\right|=\frac{1}{n}\left|\left\{ i:\mu\left(\phi(x,b_{i})\right)\geq k/N\right\} \right|.\]

It follows that :\[
\left|\frac{1}{N}\sum_{k=0}^{N-1}\tilde{\lambda}\left(C_{k/N}\right)-\frac{1}{n}\sum_{i=1}^{n}\mu(\phi(x,b_{i}))\right|\leq\frac{1}{N}\leq\epsilon.\]

Now, for $k\leq N$ let $\Theta_{k}(x)$ be the formula that says
{}``There are at least $k$ values of $i$ for which $\models\phi(x,b_{i})$
holds''. Then (looking at the Venn diagram generated by the sets
$\phi(x,b_{i})$ and counting each time each region appears in both
sums) we see that\[
\frac{1}{n}\sum_{k=1}^{n}\mu\left(\Theta_{k}(x)\right)=\frac{1}{n}\sum_{i=1}^{n}\mu\left(\phi(x,b_{i})\right).\]

Call $P'$ the value of those two sums.

By the construction of $\tilde{\lambda}$, we have the inclusions
$B_{k/n+\epsilon}(M)\subseteq\Theta_{k}(M)\subseteq B_{k/n-\epsilon}(M)$,
and $fsg$ for $\mu$ implies that $\mu(B_{k/n+\epsilon})\leq\mu(\Theta_{k}(x))\leq\mu(B_{k/n-\epsilon})$. 

Then, choosing $l$ such that $l/n\leq2\epsilon$,\[
\frac{1}{n}\sum_{k=1}^{n}\mu\left(B_{k/n+l/n}\right)\leq P'\leq\frac{1}{n}\sum_{k=1}^{n}\mu\left(B_{k/n-l/n}\right).\]

The difference between the two sums to the right and to the left of
$P'$ is at most $8\epsilon$. We may assume that $n$ was choosen
large enough so that $|R-\frac{1}{n}\sum_{k=1}^{n}\mu(B_{k/n})|\leq\epsilon.$
This latter sum satisfies the same double inequality as $P'$. Therefore
$|P'-R|\leq8\epsilon$. Putting everything together, we see that $|P-R|\leq11\epsilon$.

As $\epsilon$ was arbitrary, we are done.
\end{proof}

\section{Generic compact domination}
In \cite{NIPI} the authors introduced the notion of ``domination" or control 
of a 
type-definable set $X$ by a compact space $C$ equipped with a measure (or ideal) $\mu$: namely there is a ``definable" surjective function $\pi:X\to C$ such that for every (relatively) definable subset $Y$ of $X$, for almost all $c\in C$ in the sense of $\mu$, either $\pi^{-1}(c)\subseteq Y$ or $\pi^{-1}(c)\cap Y = \emptyset$. Here of course $X$, $\pi$ are defined over a fixed set $A$ of parameters, and $Y$ is definable with arbitrary parameters. There was also a ``group version" where $X = G$ is a (type)-definable group, $C$ is a compact group, $\pi$ a homomorphism, and $\mu$ is Haar measure on $C$. 

In this section we consider a weaker version of compact domination where $X$ is replaced by a suitable space of ``generic" types, and we expand on and correct some results which had appeared in a first version of \cite{NIPII}. 

We view this weak domination as a kind of measure-theoretic weakening of the finite equivalence relation theorem. Let us begin by explaining this interpretation.  If $T$ is a stable theory, and $p(x)\in S(A)$ a type, then the finite equivalence relation theorem states that the set of global nonforking extensions of $p$ is in one-one correspondence with the set of extensions of $p$ over $acl(A)$. Namely if $\cal P$ is the family of global nonforking extensions of $p$ then the restriction map $\pi$ taking $p'\in {\cal P}$ to $p'|acl(A)$ is a bijection. In fact these two sets, ${\cal P}$ and the set $C$ of extensions of $p$ over $acl(A)$, are both topological (Stone) spaces and the restriction map $\pi$ is a homeomorphism. 
We consider a more general situation, where $T$ has $NIP$ and $acl(A)$ is replaced by $bdd(A)$. Assuming $p$ does not fork over $A$, we again have the nonempty space ${\cal P}$ of global nonforking extensions of $p$, the space $C$ of extensions of $p$ over $bdd(A)$ and $\pi:{\cal P} \to C$. 
Now $C$ is clearly a homogeneous space for the compact Lascar group $Gal_{KP}$ over $A$, and hence is equipped with a unique $Gal_{KP}$-invariant (normalized) Borel measure, $h$ say. In this situation a weaker statement than $\pi$ being a homeomorphism is that ``${\cal P}$ is dominated by $(C,\pi,h)$": for every clopen subset $X$ of ${\cal P}$  (i.e., $X$ is given by a formula over ${\bar M}$), for almost all $c\in C$ in the sense of $h$, either $\pi^{-1}(c)\subseteq X$ or $\pi^{-1}(c)\cap X = \emptyset$. We will point out that this domination statement is equivalent to $p$ having a unique extension to a global $A$-invariant measure. In particular it will hold when $p(x)$ has $fsg$. In fact all this holds with a Keisler measure $\mu$ over $A$ in place of the complete type $p$. Subsequently we consider appropriate group versions.

We will again be assuming that $T$ has $NIP$ but it is not always needed.

\begin{Lemma} Let $\mu_{x}$ a Keisler measure over $A$. Then there is a unique Keisler measure $\mu'_{x}$ over $bdd(A)$ which extends $\mu_{x}$ and is 
$Aut(bdd(A)/A)$-invariant.
\end{Lemma}
\begin{proof} We identify a Keisler measure on $bdd(A)$ with a regular probability measure on
$S_{x}(bdd(A))$.  Likewise identify $\mu$ with a measure on $S(A)$. Now for each $p(x)\in S(A)$, there is clearly a unique $A$-invariant Keisler measure on $S(bdd(A))$ extending $p$, which we call $\mu_{p}$: the space of extensions of $p$ over $bdd(A)$ is, as mentioned above, a homogeneous space for the compact Lascar group $Aut(bdd(A)/A)$ so has a unique $A$-invariant measure, which is precisely $\mu_{p}$. Now define $\mu'$ as follows: for a Borel set $B$ over $bdd(A)$, put
$\mu'(B) = \int_{p\in S(A)} \mu_{p}(B)d\mu$. $\mu'$ is clearly $A$-invariant and we leave to the reader to check uniqueness.

\end{proof} 

\begin{Lemma} Suppose $\mu_{x}$ is a Keisler measure over $A$ which does not fork over $A$. Then $\mu$ has a global $A$-invariant extension.
\end{Lemma}
\begin{proof} Let $\lambda_{x}$ be some global nonforking extension of $\mu$. By \cite{NIPII}, $\lambda$ is $bdd(A)$-invariant and moreover Borel definable over $bdd(A)$. Fix a formula $\phi(x,b)$. Let $Q$ be the set of complete extensions of $tp(b/A)$ over $bdd(A)$. As above $Q$ is a homogeneous space for the compact Lascar group over $A$ and inherits a corresponding measure $h$ say. Define $\mu''(\phi(x,b)) = \int_{b'\in Q} \lambda(\phi(x,b'))dh$. 
\end{proof}

\begin{Lemma} Suppose $\mu_{x}$ is a Keisler measure over $A $ which has a unique global nonforking
$A$-invariant. Then for any closed subset $B$ of $S_{x}(A)$ of positive $\mu$-measure, the localization $\mu_{B}$ of $\mu$ at $B$ also has a unique global nonforking extension.
\end{Lemma}
\begin{proof}  Let $\mu'$ be the unique global nonforking extension of $\mu$. Then $\mu'_{B}$ is clearly a nonforking extension of $\mu_{B}$. If it is not the unique one, let $\lambda$ be another nonforking extension of $\mu_{B}$. 
Define global measure $\mu''$, by  $\mu''(X) = \lambda(X).\mu(B) +  \mu'(X\cap B^{c})$  (where $B^{c}$ is the complement of $B$ in $S(A)$ and $X$ a definable set). Note that $\mu''$ also extends $\mu$ and does not fork over $A$. (If $\mu''(X) > 0$, then either $\lambda(X) > 0$ or $\mu'(X) > 0$, and either way, $X$ does not fork over $A$). However by choosing $X$ such that $\lambda(X) \neq \mu'_{B}(X)$ we see that $\mu''(X) \neq \mu(X)$, contradicting our assumption.

\end{proof}

Let us set up notations for the upcoming main theorem. We fix a set $A$ (not necessarily equal to $bdd(A)$), and let $\mu_{x}$ be a Keisler measure over $A$. We will assume that $\mu$ does not fork over $A$, and we let ${\cal P}$ be the set of global types $p(x)$ that do not fork over $A$. This is a closed subspace of $S_{x}({\bar M})$. Let $C$ be $\{p(x)|bdd(A): p\in {\cal P}\}$, a closed subspace of $S_{x}(bdd(A))$ and $\pi:{\cal P}\to C$ the restriction map.  Let $\mu'$ be the (unique) $A$-invariant extension of $\mu$ over $bdd(A)$ given by Lemma 5.1. Then $\mu'$ induces (and is determined by) a Borel probability measure on $C$ which we still call $\mu'$. 

\begin{Theorem}  $\mu$ has a unique global nonforking extension if and only if ${\cal P}$ is dominated by $(C,\pi,\mu')$. 
\end{Theorem}
\begin{proof} Assume the right hand side. Let $\lambda$ be any global $A$-invariant Keisler measure extending $\mu$ (there is one by Lemma 5.2). We will show that for any definable set $X$, the value of $\lambda(X)$ is determined, independently of $\lambda$. We may identify $\lambda$ with a measure on ${\cal P}$, and $X$ with a clopen subset of ${\cal P}$. 

Note first that $\lambda|bdd(A) = \mu'$. Hence 
\newline
(*) for any Borel subset $D$ of $C$, $\lambda(\pi^{-1}(D)) = \mu'(D)$. 

Now, given our definable set $X$, let $D_{0} = \{c\in C: \pi^{-1}(c)\cap X \neq \emptyset$ and $\pi^{-1}(c)\cap X^{c})\neq \emptyset\}$. 
Then $D_{0}\subseteq C$ is closed with $\mu'$ measure $0$ by assumption. $C\setminus D_{0}$ is the disjoint union of Borel sets $D_{1}$ and $D_{2}$, where $\pi^{-1}(c)\subseteq X$ for $c\in D_{1}$ and $\pi^{-1}(c)\cap X = \emptyset$ for $c\in D_{2}$. It is then clear that $\lambda(X) = 
\lambda(\pi^{-1}(D_{1}))$ which equals $\mu'(D_{1})$ by (*). So $\lambda$ is determined.

For the converse, assume that $\mu$ has a unique global $A$-invariant extension, say $\lambda$. It is clear that $\lambda|bdd(A) = \mu'$ and that $\lambda$ must be the unique global nonforking extension of $\mu'$. Also $\cal P$ coincides with the set of global types $p(x)$ which do not fork over $bdd(A)$. HENCE we may assume that $A = bdd(A)$, and so $\mu = \mu'$. 
If the domination statement fails, there is a closed subset $D$ of $C$ of positive $\mu$-measure and formula $\phi(x)$ over ${\bar M}$, such that 
$\pi^{-1}(c)$ intersects both $\phi(x)$ and $\neg\phi(x)$ for all $c\in D$. Hence for every 
$p(x)\in D$, both $p(x)\cup\{\phi(x)\}$ and $p(x)\cup\{\neg\phi(x)\}$ do not fork over $A$. 
Let $\mu_{D}$ be the localization of $\mu$ at $D$. 
\newline
{\em Claim.} There are $\nu_{1}, \nu_{2}$, global nonforking extensions of $\mu_{D}$ such that $\nu_{1}(\phi(x)) = 1$ and
$\nu_{2}(\phi(x)) = 0$. 
\newline
{\em Proof of claim.}  Consider the fragment $G$ generated by the partial types over $bdd(A)$, $\phi(x)$ and the set $\Psi$ of formulas $\psi(x)$ (over 
${\bar M}$) which fork over $bdd(A)$. For each $p\in D$, let $r_{p} = p(x)\cup \{\phi(x)\} \cup \{\neg\psi(x):\psi\in \Psi\}$. Then $D' = \{r_{p}:p\in D\}$ is a closed subset of $S(G)$ and $f: D \to D'$ defined by $f(p) = r_{p}$ is a homeomorphism. Using $f$ to define a measure $\nu_{1}'$ supported on $D'$ gives an extension of $\mu$ which assigns $1$ to $\phi(x)$. Any extension of $\nu_{1}'$ to a global Keisler measure $\nu_{1}$ is a nonforking extension of $\mu$ assigning $1$ to $\phi(x)$.

Likewise we find $\mu_{2}$. 

\vspace{2mm}
\noindent
So the claim is proved and gives a contradiction. This completes the proof of the Theorem.

\end{proof}

Note that a version of Theorem 5.4 also holds, where we take instead $C$ to be the space of extensions over $bdd(A)$ of $\mu$-weakly random types $p(x)\in S(A)$, and ${\cal P}$ to be the space of global nonforking extensions of the types in $C$. This version is of course close to the ``finite equivalence theorem" analogy, and follows from Theorem 5.4 as stated. 

\vspace{5mm}
\noindent
Finally in this section we return to definable groups. The relevant uniqueness statement will be something like 4.3(ii).
The domination statement will be roughly the domination of a suitable family of ``generic" types by $G/G^{00}$ with its Haar measure. We start by tying up a few loose ends.
\begin{Lemma}  Let $G$ be a definable group. Suppose there is a global type $p$ of $G$ with $Stab_{l}(p) = G^{00}$. Then there is a global left $G^{00}$-invariant measure $\mu$ on $G$ which lifts (extends) Haar measure on $G/G^{00}$
\end{Lemma}
\begin{proof}  Let $h$ be Haar measure on $G/G^{00}$ which we can think of as a Keisler measure on a suitable fragment $F$ (in fact the fragment consisting of the preimages of closed sets under $\pi:G\to G/G^{00}$). 
Let $p(x)$ be as given by the assumptions. We may assume that $p$ concentrates on $G^{00}$. Note that $Stab(ap) = G^{00}$ for every translate $ap$ of $p$. In particular for $c\in G/G^{00}$ there is a unique translate of $p$ by some $a$ in the coset $c$, so we just write it $cp$. Note that for each definable subset $X$ of $G$, and $g\in G^{00}$ we have
that $X\triangle gX \notin cp$ for all $c$. It follows as in the proof of 5.4 that $h$ extends to a Keisler measure 
$\mu$ over the fragment generated by $F$ and all definable sets $X\triangle gX$  ($g\in G^{00}$), such that $\mu(X\triangle gX) = 0$ for all such $X,g$, and so to a global Keisler measure which is $G^{00}$-invariant. 
\end{proof}

\begin{Lemma} Suppose $\mu$ is a global Keisler measure on $G$ which is (left) $G^{00}$-invariant. Then $\mu(X\triangle gX) = 0$ for all definable $X$ and $g\in G^{00}$. In particular for all $\mu$-weakly random global $p$, 
$Stab_{l}(p) = G^{00}$.
\end{Lemma}
\begin{proof}
This is a kind of group version of the fact that if a global Keisler measure is $bdd(A)$-invariant (does not fork over $bdd(A)$) then $\mu(X\triangle X') = 0$ for any $bdd(A)$-conjugate $X'$ of $X$. The proof of the latter was easy, but there does not seem to be such a straightforward proof of the new lemma.
 We have to prove that $Stab_{l}(p) = G^{00}$ for each $\mu$-weakly random global type $p$. 
 Passing to a bigger monster model ${\bar M'}$, and arguing as in 5.8 of \cite{NIPII}, $\mu$ has an extension to a definable left $G^{00}$-invariant measure $\mu'$ over ${\bar M'}$. But then clearly there is a small $M_{0}'$ such that for all $g\in G({\bar M'})$, $g\mu'$ does not fork over $M_{0}'$. Now our $\mu$-weakly random type $p$ of $\mu$ extends to a $\mu'$-weakly random type $p'$. By what we have just seen, $p'$ is left $f$-generic (every left translate does not fork over a fixed $M_{0}'$). By Proposition 5.6(i) of \cite{NIPII}, $Stab_{l}(p') = G^{00}({\bar M'})$. It follows that $Stab_{l}(p) = G^{00}$. 
\end{proof}

\begin{Proposition} Let $G$ be a definable group. Then the following are equivalent:
\newline
(i) There is a unique left $G^{00}$-invariant global Keisler measure of $G$ lifting Haar measure on $G/G^{00}$,
\newline
(ii) Let $\cal P$ be the family of global complete types of $G$ such that $Stab_{l}(p) = G^{00}$. Let $\pi$ be the 
canonical surjective map from ${\cal P}$ to $G/G^{00}$, and $h$ Haar measure on $G/G^{00}$. Then ${\cal P}$ is nonempty and is dominated by $(G/G^{00}, \pi, h)$.
\end{Proposition}
\begin{proof} (i) implies (ii): This is like LHS implies RHS in 5.4. The nonemptiness of ${\cal P}$ is given by the previous Lemma. Write $C$ for $G/G^{00}$. For $c\in C$, let ${\cal P}_{c}$ be those members of ${\cal P}$ which concentrate on $\pi^{-1}(c)$. Suppose for a contradiction that for some definable subset $X$ of $G$, the (closed) subset $D$ of $C$ consisting of $c$ such that both ${\cal P}_{c}\cap X$ and ${\cal P}_{c}\cap X^{c}$ is non empty, has positive Haar measure. Then, as in proofs of 5.4 and 5.5, $h$ lifts to
two global left $G^{00}$-invariant Keisler measures, one giving $D\cap X$ positive measure (in fact that of $D$) and the other giving it $0$ measure. Contradiction.
\newline
(ii) implies (i). The existence of $\mu$ is given by 5.5. So let us fix global left $G^{00}$-invariant $\mu$ which lifts Haar measure $h$ on $G/G^{00}$. Let ${\cal P'}$ be the set of family of $\mu$-weakly random global types. By Lemma 5.6, ${\cal P'}$ is a nonempty subset of ${\cal P}$ which clearly maps onto $G/G^{00}=C$ under $\pi$. 
So now write $\pi$ for the map ${\cal P'} \to C$. The assumption (ii) implies that also ${\cal P'}$ is dominated by $(C,\pi,h)$. As in the proof of RHS implies LHS of Theorem 5.4, we conclude that for any definable set $X$, $\mu(X) = h(D)$ where $D = \{c\in C:\pi^{-1}(c)\subseteq X\}$. So $\mu$ is determined. 
\end{proof}

\vspace{5mm}
\noindent
By 4.3, the previous proposition applies to any definable group with $fsg$. In fact 
the class ${\cal P}$ of global types with stabilizer $G^{00}$ can clearly be replaced by the subclass ${\cal P}_{gen}$ of
global {\em generic} types $p$ of $G$. (Here generic is in the sense of \cite{NIPII}.)
Hence we have ``generic compact domination" for $fsg$ groups:
\begin{Proposition} Suppose $G$ has $fsg$. Let ${\cal P}_{gen}$ be the space of global generic types of $G$, $\pi:G\to G/G^{00}$ as before and $h$ Haar measure on $G/G^{00}$. Then ${\cal P}_{gen}$ is dominated by $(G/G^{00},\pi,h)$.
\end{Proposition}

Finally we point out that under an additional hypothesis on $G$, we can slightly strengthen the domination statement.
Let us fix a definable group $G$, $\pi:G\to G/G^{00}=C$, and definable subset $X$ of $G$. We will 
say that $X$ is {\em left generic in $\pi^{-1}(c)$} if finitely many left translates of $X$ by elements of $G^{00}$ cover $\pi^{-1}(c)$. 

We will be interested in the following hypothesis on an $fsg$ group $G$:
\newline
(H): Let $X\subseteq G$ be definable. Then $X$ is generic in $G$ if and only if for some small model $M$ every left translate $gX$ of $X$ does not divide over $M$.

\vspace{2mm}
\noindent
The left hand side implies the right hand side in any $fsg$ group. We do not know an example of an $fsg$ group $G$ where (H) fails. It is true in any $o$-minimal expansion of $RCF$. 

\begin{Lemma} Suppose that the $fsg$ group $G$ satisfies (H). Let $X$ be a definable subset of $G$, and $c\in C$. The following are equivalent:
\newline
(i) $X$ is generic in $\pi^{-1}(c)$,
\newline
(ii) $X\in p$ for every $p\in {\cal P}_{gen}$ concentrating on $c$,
\newline
(iii) for some definable set $Y$ containing $\pi^{-1}(c)$, $Y\setminus X$ is not generic in $G$.
\end{Lemma}
\begin{proof} Without loss of generality let $c$ be the identity of $G/G^{00}$. 
\newline
(i) implies (ii) holds without even assuming (H) as the stabilizer of any generic type is $G^{00}$. (See \cite{NIPI} or \cite{NIPII}.)
\newline
(ii) implies (iii) also holds without assuming (H): By (ii), $``x\in G^{00}" \cup ``x\notin X" \cup \{\neg \psi:\psi$ over ${\bar M},\psi$ nongeneric\} is inconsistent, so by compactness, for some definable $Y$ containing $G^{00}$, $Y\setminus X$ is nongeneric (we use here that the set of nongenerics is an ideal).
\newline
(iii) implies (i): Let $Z = Y\setminus X$. Let $M$ be a model over which $X$ and $Y$ are defined. By (H) there is some $M$-indiscernible sequence
$(a_{i}:i<\omega)$ of elements of $G$ such that $\cap_{i=1,..,n}a_{i}Z = \emptyset$ for some $n$. Let $g_{i} = a_{1}^{-1}a_{i}$ for $i=1,..,n$. So
$\cap_{i=1,..,n}g_{i}Z = \emptyset$. All elements of $a_{i}$ are in the same coset of $G^{00}$, hence $g_{i}\in G^{00}$ for $i=1,..,n$. As each $g_{i}Y$ contains $G^{00}$ it follows that $g_{1}X \cup... \cup g_{n}X \supseteq G^{00}$.

\end{proof}

\begin{Corollary} Suppose that $G$ is a group with $fsg$ which satisfies (H), and $\pi:G\to G/G^{00}$.  
Then for any definable subset $X$ of $G$, for almost all $c\in G/G^{00}$ in the sense of Haar measure, either $X$ is generic in $\pi^{-1}(c)$ or $\neg X$ is generic in $\pi^{-1}(c)$.
\end{Corollary}
\begin{proof} Clear.
\end{proof}

\section{Borel measures over standard models}
In this section we give a rich source of smooth measures in the case of theories of $o$-minimal expansions of $\R$, as well as $Th(\Q_{p})$. If $M_{0}$ is the standard model, $V\subseteq M_{0}^{n}$ is definable, and $\mu^{*}$ is a  Borel probability measure on the topological space $V$, then by restricting $\mu^{*}$ to definable sets, we have a Keisler measure which we  call $\mu$, over $M_{0}$. We will show that any such $\mu$ is smooth: has a unique extension to a Keisler measure $\mu'$ over a saturated model. It follows in particular that $\mu'$ will be ``definable" (\cite{NIPI}), from which one can easily obtain ``approximate definability" of $\mu^{*}$ in the sense of Karpinski and Macintyre \cite{Karpinski-Macintyre}, thereby considerably generalizing their results on approximate definability of the real and $p$-adic Haar measures
on unit discs. 

For now, $T$ is an arbitrary complete theory. It is convenient to formally weaken the notion of a Keisler measure by allowing values in $[0,r]$ for some $r$, but of course maintaining finite additivity.  Sometimes we may say that the Keisler measure $\mu$ is ON the definable set $X$ if $\mu(X^{c})= 0$. 

\begin{Definition} Let $\mu_{x}$ be a Keisler measure over a model $M$. We will say that $\mu$ is {\em countably additive} over $M$,
if whenever $X$ is definable over $M$, $Y_{i}$ are definable over $M$ for $i<\omega$ and pairwise disjoint and 
$X(M)$ is the union of the $Y_{i}(M)$, then $\mu(X) = \sum_{i<\omega}\mu(Y_{i})$. 
\end{Definition} 

\begin{Remark} (i) Of course the definition depends on $M$. When $M$ is $\omega$-saturated, any Keisler measure over $M$ is countably additive, because if $X(M)$ is the union of the $Y_{i}(M)$ then by compactness it will be a finite subunion. 
\newline
(ii) If, as above, $M_{0}$ is a structure whose underlying set is a topological space $X$ say, and such that all definable subsets of the universe are Borel, THEN any Borel measure $\mu^{*}$ on $X$ such that $\mu^{*}(X) \neq\infty$ induces a countably additive Keisler measure over $M_{0}$ (on $X$), by restricting to definable sets. 
\end{Remark}

\begin{Theorem} Let $M_{0}$ be either an $o$-minimal expansion of $(\R,+,\cdot)$, or the structure $(\Q_{p},+,\cdot)$. Let $V$ be a definable set in $M_{0}$, and $\mu$ a countably additive Keisler measure on $V$ over $M_{0}$. THEN $\mu$ is smooth. That is, $\mu$ has a unique extension to a global Keisler measure on $V$.
\end{Theorem}
\begin{proof} We will distinguish here between the definable set $V$ (as a functor say) and the set $V(M_{0})$ of $M_{0}$-points.  The proof is by induction on the $o$-minimal/$p$-adic dimension of  $V$ (or $V(M_{0})$) which we take to be $n$. Clearly it suffices to partition $V$ into $M_{0}$-definable sets $V_{1},..,V_{k}$ and prove the proposition for $\mu_{i} = \mu|V_{i}$ for each $i$ (where we stipulate that $\mu_{i}$ is $0$ outside 
$V_{i}$). So by cell-decomposition we may assume that $V\subseteq I^{n}$ where $I$ is the closed unit interval $[0,1]$ in the $o$-minimal case, and the valuation ring in the $p$-adic case. So in fact there is no harm in assuming that $V = I^{n}$. 

Let us fix an extension $\mu'$ of $\mu$ to a global Keisler measure. And let $D$ be a definable (over ${\bar M}$) subset of $V$. We aim to show that $\mu'(D)$ is determined, namely can be computed in terms of $\mu$. 

Recall that we have the standard part map $st$ from $I^{n}({\bar M})$ to $I^{n}(M_{0})$, namely from $V({\bar M})$ to $V(M_{0})$. In both the $o$-minimal and $p$-adic cases all types over the standard model are definable (\cite{Marker-Steinhorn}, \cite{Delon}). As explained in \cite{OP} for example, this implies that for any definable in ${\bar M}$ subset $X$ of $I^{n}$, $st(X)$ is a definable set in the structure $M_{0}$. In particular $st(D)$, $st(D^{c})$, and the intersection $st(D)\cap st(D^{c})$ are definable sets in the structure $M_{0}$. Hence we can write $V(M_{0})$ as the disjoint union of $Y(M_{0})$, 
$D_{0}(M_{0})$, and $D_{1}(M_{0})$, where as the notation suggests $Y,D_{0}, D_{1}$ are definable over $M_{0}$, $Y(M_{0}) = st(D)\cap st(D^{c})$, $D_{1}(M_{0}) = st(D)\setminus Y$ and $D_{0}(M_{0}) = st(D^{c})\setminus Y$. Note also that $D_{0}$, $D_{1}$ are open $M_{0}$-definable subsets of $V$, and of course $V$ is the disjoint union of $Y$, $D_{0}$ and $D_{1}$. 

\vspace{2mm}
\noindent
{\em Claim 1.}  The $M_{0}$-definable subset $Y\cup (cl(D_{0})\cap cl(D_{1}))$ of $V$ has dimension $< n$.
\newline
{\em Proof of Claim 1.} Otherwise it contains an open $M_{0}$ definable set $U$ say. But then either $D\cap U$ or $D^{c}\cap U$ contains an open $M_{0}$-definable subset (of $V = I^{n}$). (In the $p$-adic case this is Theorem 2.2(ii) of \cite{OP}. It is well-known in the $o$-minimal case too, but formally follows from 10.3 of \cite{NIPI} for example.) But this is clearly impossble. For if, for example,  $W$ is an open $M_{0}$-definable set contained in $D$, then for all $w\in W(M_{0})$, $st^{-1}(w)\subseteq D$, so for each $w\in W(M_{0})$, $w\notin Y(M_{0})$, and 
$w\notin cl(D_{0})(M_{0})$. The claim is proved.

\vspace{2mm}
\noindent
Let $D_{2}$ be the (closed, $M_{0}$-definable) set $Y\cup (cl(D_{0})\cap cl(D_{1}))$. Let $\mu_{2}$ be $\mu|D_{2}$. Namely $\mu_{2}$ agrees with $\mu$ on $M_{0}$-definable subsets of $D_{2}$ and is $0$ on the complement of $D_{2}$. Likewise define $\mu'_{2}$ to be equal to $\mu'$ on definable subsets of $D_{2}$ and $0$ on the complement of $D_{2}$. Then as $\mu_{2}$ is still countably additive, we see, by induction hypothesis and Claim 1, that $\mu'_{2}$ is the unique global extension of $\mu_{2}$. In particular we have:
\newline
{\em Claim 2.}  $\mu'_{2}(D) = \mu'(D\cap D_{2})$ is determined. 

\vspace{2mm}
\noindent
{\em Claim 3.} Let $D_{3}(M_{0})$ be an open $M_{0}$-definable neighbourhood of the closed set $D_{2}(M_{0})$. Then $D\setminus D_{3} = D_{1}\setminus D_{3}$, hence $\mu'(D\setminus D_{3}) = \mu(D_{1}\setminus D_{3})$. 
\newline
{\em Proof of Claim 3.} Let $a\in V = V({\bar M})$, and suppose $a\notin D_{3}$. So 
\newline
(*) $st(a)\notin D_{2}(M_{0})$. 
\newline
Case (i): $a\in D_{1}$. Then $st(a)\in cl(D_{1})(M_{0})$. By (*), $st(a)\notin Y$, and $st(a)\notin D_{0}$. Hence $st(a)\in D_{1}$ and we conclude that $a\in D$.
\newline
Case (ii): $a\in D_{0}$. As in Case (i) we conclude that $st(a)\in D_{0}$ hence $a\notin D$.
\newline
This proves Claim 3.

\vspace{2mm}
\noindent
{\em Claim 4.}  $\mu'(D\setminus D_{2}) = \mu(D_{1}\setminus D_{2})$.
\newline
{\em Proof.} For small $\delta > 0$ let $D_{\delta}(M_{0})$ be the $\delta$ neighbourhood of $D_{2}$. Then $\cap_{\delta}D_{\delta} = D_{2}$.
So $\mu(D_{\delta}\setminus D_{2}) \to 0$ as $\delta \to 0$, hence also $\mu'(D\cap (D_{\delta}\setminus D_{2}) \to 0$ as $\delta \to 0$. 
It follows, using Claim 3, that $\mu'(D\setminus D_{2}) = lim_{\delta\to 0}\mu'(D\setminus D_{\delta} = lim_{\delta\to 0}\mu(D_{1}\setminus D_{\delta}) = \mu(D_{1}\setminus D_{2})$.

\vspace{2mm}
\noindent
Claims 2 and 4 show that $\mu'(D)$ is determined. 
\end{proof}

\begin{Remark} (i) The key point about countably additive Keisler measures $\mu$ over the standard model is that any global extension $\mu'$ must assign $0$ to definable sets which are ``infinitesimal".
\newline
(ii) The inductive proof of Theorem 6.3  yields the following: for any definable subset $D$ of $I^{n}({\bar M})$, there is a partition of
$I^{n}$ into $M_{0}$-definable cells $V_{1},..,V_{k}$, such that for each $i$, EITHER for all $a\in V_{i}(M_{0})$ $st^{-1}(a)\cap V_{i}\subseteq D$, OR
for all $a\in V_{i}(M_{0})$ $st^{-1}(a)\cap V_{i}\cap D = \emptyset$. 

\end{Remark}


\begin{thebibliography}{99}
\bibitem{Adler} H. Adler, Introduction to theories without the independence property, to appear in Archive Math. Logic.

\bibitem{BY1} I. Ben-Yaacov, Continuous and random Vapnik-Chervonenkis classes, to appear.

\bibitem{BY2} I. Ben-Yaacov, Transfer of properties between measures and random types, preprint.

\bibitem{BY-Keisler}  I. Ben Yaacov and H.J. Keisler, Randomization of models as metric structures, preprint.

\bibitem{Delon} F. Delon, D\'efinissabilit\'e avec parametres ext\'erieurs dans $\Q_{p}$ et $\R$, Proc. AMS  106 (1989) 193-198.





\bibitem{HHM2} D. Haskell, E. Hrushovski, and D. Macpherson, {\em Stable domination and independence
in algebraically closed valued fields}, Lecture Notes in Logic 30, CUP, 2007.

\bibitem{HKP} B. Hart, B. Kim, and A. Pillay, Coordinatization and canonical bases in simple
theories, Journal of Symbolic Logic, 65(2000), 293-309.




\bibitem{NIPI} E. Hrushovski, Y. Peterzil, and A. Pillay, Groups, measures and the $NIP$, Journal AMS, 21 (2008), 563-596. 

\bibitem{NIPII} E. Hrushovski and A. Pillay, On $NIP$ and invariant measures, preprint second version Jan. 2009, submitted.

\bibitem{Hrushovski-Loeser} E. Hrushovski and F. Loeser, Non-archimedean tame topology and stably dominated types, in preparation. 

\bibitem{Karpinski-Macintyre} M. Karpinski and A. J. Macintyre, Approximating volumes and integrals in $o$-minimal and $p$-minimal theories, in {\em Connections between Model Theory and Algebraic and Analytic Geometry}, Quaderni di matematica, vol 6, Seconda Universita di Napoli, 2000.

\bibitem{Keisler1} H. J. Keisler, Measures and forking, Annals of Pure and Applied Logic 45 (1987), 119-169.

\bibitem{Keisler2} H. J. Keisler, Choosing elements in a saturated model, {\em Classification Theory, Proceedings, Chicago 1985}, ed. J. Baldwin, Lecture Notes in Math. 1292, 1987. 






\bibitem{Lascar-Pillay} D. Lascar and A. Pillay, Hyperimaginaries and automorphism groups, Journal
of Symbolic Logic, 66(2001), 127-143.

\bibitem{Marker-Steinhorn} D. Marker and C. Steinhorn, Definability of types in $o$-minimal theories, JSL, 59 (1994), 185-198.


\bibitem{OP} A. Onshuus and A. Pillay, Definable groups and compact $p$-adic Lie groups, Journal LMS, 78(2008). 





\bibitem{Pillay-book} A. Pillay, {\em Geometric Stability Theory}, Oxford University Press 1996.


\bibitem{Poizat} B. Poizat, {\em A Course in Model Theory}, Springer 2000.



\bibitem{Shelah783} S. Shelah, Dependent first order theories, continued, to appear in Israel J. Math.

\bibitem{Shelah} S. Shelah, {\em Classification theory}, 2nd edition, North Holland, 1990.

\bibitem{Simon} P. Simon, Th\'eories NIP, M.Sc. thesis, Paris VII.

\bibitem{Usvyatsov} A. Usvyatsov, On generically stable types in dependent theories, JSL,
74 (2009), 216-250



\bibitem{Vapnik-Chervonenkis} V.N. Vapnik and A.Y. Chervonenkis, On the uniform convergence of relative frequencies of events to their probabilities, Theory Probab. Appl., 16 (1971), 264-280.

\bibitem{Wagner} F. Wagner, {\em Simple theories}, Kluwer.

\end{thebibliography}
\end{document}